\documentclass[a4paper,leqno,12pt]{amsart}
 \usepackage[usenames]{color}
 \usepackage{amssymb,amsthm,amsmath,amscd}
\usepackage{ascmac}
 \usepackage{fancybox}
 \usepackage{ulem}
 \usepackage[all]{xy}
 \usepackage{comment}
 \input amssym.def
 \input amssym.tex
\newcommand{\Stab}{\mathrm{Stab}}
 \newcommand{\Ker}{\mathrm{Ker}}
 
  \newcommand{\Aut}{\mathrm{Aut}}
 \newcommand{\Mult}{\mathrm{mult}} 
 
  \newcommand{\Fix}{\mathrm{Fix}}
   
 \newcommand{\Pic}{\mathrm{Pic}}

 \newcommand{\PROJ}{\mathbb{P}}
 
 \newcommand{\Red}[1]{\textcolor{red}{#1}}
 \newcommand{\Z}{\mathbb{Z}}
 
 \newcommand{\C}{\mathbb{C}}
 
  \newcommand{\Ha}[1]{\widehat{#1}}
  \newcommand{\Ov}[1]{\overline{#1}}
 \newcommand{\Ti}[1]{\widetilde{#1}}
  \newcommand{\Ch}[1]{\check{#1}}
  \newcommand{\Bb}[1]{\mathbb{#1}}
  \newcommand{\Fr}[1]{\mathfrak{#1}}
   \newcommand{\Rm}[1]{\mathrm{#1}}
    
 \newcommand{\Al}{\alpha} 
 \newcommand{\Ga}{\Gamma} 
 \newcommand{\PHI}{\varphi}
 \newcommand{\Eps}{\varepsilon}
 \newcommand{\Sh}{\sharp}

 \newcommand{\RB}{\overline{R}}
 
 \newcommand{\RT}{\widetilde{R}}
 \newcommand{\DT}{\widetilde{\mathfrak{d}}}
 \newcommand{\ST}{\widetilde{S}}
 \newcommand{\PB}{\overline{P}} 
 \newcommand{\PT}{\widetilde{P}}
    
 \newcommand{\SigmaT}{\tilde{\sigma}}

  \newcommand{\GAB}{\overline{\Gamma}}
   
  \newcommand{\GAT}{\widetilde{\Gamma}}
   \newcommand{\KT}{\widetilde{K}}
  \newcommand{\KB}{\overline{K}}
    \newcommand{\CB}{\overline{C}}
 
  \newcommand{\PHIT}{\widetilde{\varphi}}  
    \newcommand{\PHIB}{\overline{\varphi}}

\newtheorem{prop}{Proposition}[section]
 \newtheorem{thm}[prop]{Theorem}
 \newtheorem{lem}[prop]{Lemma}

\theoremstyle{remark}
\newtheorem{define}[prop]{Definition}
\newtheorem{rmk}[prop]{Remark}
\numberwithin{equation}{section} 
 \makeatletter
\@namedef{subjclassname@2020}{2020 Mathematical Subject Classification}
\makeatother
\begin{document}
\baselineskip=18pt
 \title{Bounds for the order of automorphism groups of cyclic covering fibrations of a ruled surface}
\author{Hiroto Akaike}

 
 
\maketitle

\begin{abstract}
We study the order of automorphism groups of cyclic covering fibrations of a ruled surface.
Arakawa and later Chen studied it for hyperelliptic fibrations and gave the upper bound.
The purpose of present paper is to pursue the analog for cyclic covering fibrations. 
\end{abstract}

\section*{Introduction}
Let $f:S\to B$ be a surjective morphism from a complex smooth projective surface $S$ 
to a smooth projective curve $B$ with connected fibers. 
We call it a {\it fibration} of genus $g$ when a general fiber is a curve of genus $g$.
A fibration is called {\it relatively minimal}, when any $(-1)$-curve is not contained in fibers.
Here we call a smooth rational curve $C$ with $C^2=-n$ a $(-n)$-curve.
A fibration is called {\it smooth} when all fibers are smooth, {\it isotrivial} when all of the 
smooth fibers are isomorphic, {\it locally trivial} when it is smooth and isotrivial.

An automorphism of the fibration $f:S \to B$ is a pair of automorphisms $(\kappa_{S},\kappa_{B})\in \Aut(S)\times  \Aut(B)$
satisfying $ f \circ \kappa_{S} =  \kappa_{B} \circ  f $, that is, the diagram 
\begin{equation*}
\xymatrix{
S\ar[r]^-{\kappa_{S}}\ar[d]_-{f}\ar@{}[rd]|{\circlearrowright}&S\ar[d]^-{f}\\
B\ar[r]_-{\kappa_{B}}&B
}
\end{equation*}
is commutative.
We denote  by $\Aut (f)$ the group of all automorphisms of $f$.
Arakawa  \cite{Ara} and later Chen \cite{Chen} studied it for hyperelliptic fibrations and gave the upper bound 
of the order of a finite subgroup of $\Aut (f)$ in terms of $g$, $g(B)$ (the genus of $B$) and $K_f^2$ (the square of the relative canonical bundle).
The purpose of the present paper is to pursue the analog for primitive cyclic covering fibrations of type 
$(g,0,n)$ introduced by Enokizono in \cite{Eno}.
Roughly, they are $n$-fold cyclic branched coverings of ruled surfaces and hyperelliptic fibrations are nothing more than those with $n=2$.

Let $f:S \to B$ be a primitive cyclic covering fibration of type $(g,0,n)$ and $G$ a finite subgroup of $\Aut(f)$.
Then $G$ can be expressed as the extension of its horizontal part $H$ by its vertical part $K$, that is, 
$1\to K\to G\to H \to 1$ (exact). Here $H$ is a subgroup of $\Aut(B)$ consisting of $\kappa_{B}$ such that 
one can find a $\kappa_{S}\in \Aut(S)$ satisfying $(\kappa_{S},\kappa_{B})\in G$, while $K$ naturally acts on fibers of $f$ as a subgroup of $\Aut(S/B)$.
We can study $H$ by using known results for the automorphism groups of curves. 
Hence it suffices to manage $K$.
Note that we have the canonical cyclic subgroup of $\Aut(S/B)$ of order $n$, the covering transformation group, with 
a (singular) ruled surface as the quotient by its action.
We may assume that $K$ contains it.
Then $K$ induces a subgroup $\KT$ of the automorphism group of the ruled surface preserving the branch locus 
and our task is reduced to estimating the order of $\KT$.
For this purpose, we use the localization of the invariant $K_f^2$.
It is known that $K_f^2$ can be localized to fibers of $f$ (Lemma~1.7 of \cite{Eno}, or Lemma~\ref{lem1.7}).
We refine the localized $K_f^2$ further by using the quantity defined at a point on a fiber, 
and obtain the new expression of it in Proposition~\ref{loc'nofK}.
Then, we estimate the order of $\Ti{K}$ from above with the localized $K_f^2$ multiplied by an explicit function in $g$ and $n$ (Propositions~\ref{summarize}).

Now, the main results can be stated as follows.

\begin{thm}[Theorem~\ref{up.bd.1}]
Let $f:S \to B$ be a non-isotrivial primitive cyclic covering fibration of type $(g,0,n)$ 
with $n\geq 3$, $g \geq (3n-2)(n-1)/2$ if $n\geq 5$, $g \geq 21$ if $n=4$ and $g \geq 16$ if $n=3$.
Put 
\begin{align*}
\mu_{g,n}:=\left(1+\frac{1}{n-1}\right)^{2}\frac{2g+n-1}{g-1},\quad 
\mu_{g,n}':=\frac{2n^2(g+n-1)(2g+n-1)}{\left((n+1)g+n-1\right)(n-1)(g-1)}
\end{align*}
Let $G$ be a finite subgroup of $\Aut (f)$.
Then it holds
\[
  \Sh G \leq \begin{cases}
    6(2g(B)-1)\mu_{g,n}K_{f}^2 & (g(B) \geq 1 ) \\
     5\mu_{g,n}'K_{f}^2 & (g(B)=0) .
  \end{cases}
\]
\end{thm}

\begin{thm}[Theorem~\ref{6/16.cor}]
Let $f:S \to \Bb{P}^1$ be a non-locally trivial primitive cyclic covering fibration of type $(g,0,n)$ with 
with $n\geq 3$, $g \geq (3n-2)(n-1)/2$ if $n\geq 5$, $g \geq 21$ if $n=4$ and $g \geq 16$ if $n=3$.
Then it holds
\begin{align*}
\Sh \Aut (S/B) \leq \left(1+\frac{1}{n-1}\right)^{2}\frac{2g+n-1}{2g-2} K_{f}^2.
\end{align*}
\end{thm}
In Sections~1 and 2, we recall basic properties of primitive cyclic covering fibrations mainly due to \cite{Eno}.
Among other things, we recast Section~5 of \cite{Eno} and consider the refined localization of $K_f^2$.
In Section~3, we study automorphisms of primitive cyclic covering fibrations of type $(g,0,n)$ to reduce the problem to that on a ruled surface.
Especially we discuss whether it is possible to perform a sequence of equivariant blowing-downs (with respect to $\KT$) in order to get 
a geometrically ruled surface on which the branch locus has only singular points of standardized multiplicities.
In Section~4, we estimate the order of $\Ti{K}$.
In Section~5, we show our main result, Theorem~\ref{up.bd.1}.
Section~6 is devoted to construct an example 
that shows our estimate of Theorem~\ref{6/16.cor} is almost optimal.

\vspace{3\baselineskip}

\textbf{Acknowledgememt}

The author is deeply grateful to Professor Kazuhiro Konno for his valuable advices and supports.
The author is particularly grateful to Professor Osamu Fujino and Associate Professor Kento Fujita for their continuous supports.  
The author also thanks to Doctor Makoto Enokizono for his precious advices.
The research is supported by JSPS KAKENHI No. 20J20055.

\tableofcontents


\section{Primitive cyclic covering fibrations}
We recall basic properties of primitive cyclic covering fibrations, most of which can be found in \cite{Eno}.

\begin{define}
Let $f:S \to B$ be a relatively minimal fibration of genus $g\geq2$.
We call it a {\it primitive cyclic covering fibration} of type $(g,0,n)$, when 
there are  a $($not necessarily relatively minimal $)$ fibration $\tilde{\PHI}:\PT \to B$ of 
genus $0$ ({\it i.e.} ruled surface) and a classical $n$-cyclic covering
$$
\tilde{\theta}:\ST=
\mathrm{Spec}_{\PT}\left(\bigoplus_{j=0}^{n-1} \mathcal{O}_{\PT}(-j\DT)\right)\to\PT
$$
branched over a smooth curve $\RT \in |n\DT|$ for some 
$n\geq2$ and $\DT \in \Pic(\PT)$ such that $f$ is the relatively minimal model of 
$\tilde{f}=\tilde{\PHI}\circ\tilde{\theta}$.
\end{define}

In addition, we employ the following notation. 
We denote by $\Sigma=\langle \tilde{\sigma} \rangle$ the covering transformation group of $\tilde{\theta}$ and by 
$\PHI: P\to B$ a relatively minimal model of $\tilde{\PHI}:\PT \to B$ with the natural contraction map 
$\tilde{\psi}: \PT \to P$. 
Furthermore, 
$\widetilde{F}$, $F$, $\widetilde{\Gamma}$ and $\Ga$ will denote general fibers of 
$\tilde{f}$, $f$,  $\tilde{\PHI}$ and $\PHI$, respectively.
To specify it is a fiber over $p\in B$, we write $F_p$, etc.

\begin{rmk}
\label{basic}
We note important properties of primitive cyclic covering fibrations in the following, which can be found in \cite{Eno}. 
\begin{itemize}
\item If $\sigma$ is the automorphism of $S$ over $B$ induced by a generator $\SigmaT$ of $\Sigma$, 
then we can assume that the natural morphism $\ST\to S$ is a minimal succession of blowing-ups 
that resolves all isolated fixed points of $\sigma$.
\item $\Ti{\varphi}$-vertical components of $\Ti{R}$ is a disjoint union of non-singular rational curves each of which is a $(-an)$-curve for some positive integer $a$.  
\item Arbitrary $\Ti{\varphi}$-vertical $(-1)$-curve and $\Ti{R}$ meet.
\end{itemize}
\end{rmk}

From now on, we let $f:S \to B$ be a  primitive cyclic covering fibration of type $(g,0,n)$ and freely use the above notation and convention.

Since the restriction map 
$\tilde{\theta}|_{\widetilde{F}}:\widetilde{F} \to \widetilde{\Gamma}\simeq \mathbb{P}^1$ is a classical 
$n$-cyclic covering branched over $\RT\cap\widetilde{\Ga}$, 
the Hurwitz formula for $\tilde{\theta}|_{\widetilde{F}}$ gives us
\begin{equation}
r:=\RT\widetilde{\Ga}=\frac{2\bigl(g+n-1\bigr)}{n-1}.
\end{equation}
From $\RT \in |n\DT|$, it follows that $r$ is a multiple of $n$.

Since $\tilde{\psi}: \PT \to P$ is a composite of blowing-ups, we can write 
$\tilde{\psi}=\psi_{1}\circ\cdots\psi_{N}$, 
where $\psi_{i}:P_{i}\to P_{i-1}$ denotes the blowing-up 
at $x_{i} \in P_{i-1}\;(i=1,\cdots, N)$, $P_{0}=P$ and $P_{N}=\PT$.
We define a reduced curve $R_{i}$ inductively as $R_{i-1}=(\psi_{i})_{\ast}R_{i}$ 
starting from $R_{N}=\RT$ down to $R_{0}=R$.
We call $R_{i}$ {\it branch locus on $P_{i}$}.
We also put $E_{i}=\psi_{i}^{-1}(x_{i})$ and $m_{i}=\mathrm{mult}_{x_{i}}R_{i-1}\;(i=1,\cdots, N)$.

\begin{lem}
In the above situation, the following hold for any $i=1,\cdots, N$.
\label{lem1.2}
\begin{itemize}
\item[$(1)$] Either $m_{i} \in n\Z_{\geq1}$ or $n\Z_{\geq1}+1$, where $\Z_{\geq 1}$ is the set of positive integers. 
Furthermore, $m_{i}\in n\Z_{\geq1}$
if and only if $E_{i}$ is not contained in $R_{i}$.

\item[$(2)$] $R_{i}={\psi}_{i}^{\ast}R_{i-1}-n[\frac{m_{i}}{n}]E_{i}$, where $[t]$ denotes the 
greatest integer not exceeding $t$.

\item[$(3)$] There exists a $\Fr{d}_{i} \in \Pic(P_{i})$ such that $\Fr{d}_{i}=\psi_{i}^{\ast}\Fr{d}_{i-1}-[\frac{m_{i}}{n}]E_{i}$
 and $R_{i}\sim n\Fr{d}_{i}$, $\Fr{d}_{N}=\DT$. 
\end{itemize}
\end{lem}

We say that a singular point of an $R_i$ is of type $n\mathbb{Z}$ (resp. $n\mathbb{Z}+1$) if its multiplicity 
is in $n\mathbb{Z}_{\geq 1}$ (resp. $n\mathbb{Z}_{\geq 1}+1$).

\begin{lem}
\label{up.bd.of.mult}
Let $f:S\to B$ be a primitive cyclic covering fibration of type $(g,0,n)$.
Then there is a relatively minimal model $\PHI: P\to B$ of $\tilde{\PHI}:\PT \to B$ such that 
$$
\mathrm{mult}_{x}R_{h}\leq\frac{r}{2}=\frac{g}{n-1}+1
$$
for all $x \in R_{h}$, where $R_{h}$ denotes the $\PHI$-horizontal part of $R$.
Moreover if $\mathrm{mult}_{x}R>\frac{r}{2}$, then 
$\mathrm{mult}_{x}R \in n\Z + 1$.
\end{lem}

Later, we will need to pay a special attention to the action of a finite subgroup of $\Aut (\PT/B)$ preserving $\RT$.
We show in Lemma~\ref{hyoujyunka} that there is a relatively minimal model $\PHI: P\to B$, compatible with such an action, 
which satisfies the inequality in Lemma~\ref{up.bd.of.mult}.
So we assume that $\PHI: P\to B$ always satisfies Lemma $\ref{up.bd.of.mult}$ in what follows.

\vspace{2\baselineskip}

\section{Localization of $K_f^2$}
We reconsider the argument in  Section~5 of \cite{Eno} and give localizations of $K_f^2$.
Let $f:S\to B$ be a primitive cyclic covering fibration of type $(g,0,n)$.
For an effective vertical divisor $T$ and $p \in B$, we denote the biggest subdivisor of $T$ whose support is in the fiber over $p$ by $T(p)$. 
Then $T = \sum_{p \in B}T(p)$.
We sometimes write $\sharp T$ to indicate the number of irreducible components of $T$.


First of all, let us recall the singularity index \cite{Eno}.
For any fixed $p\in B$, we consider all singular points (including infinitely near ones) of $R$ on $\Gamma_{p}$. 
For any positive integer $k$, we let $\alpha_{k}(\Ga_{p})$ be the number of singular points of multiplicity either $kn$ or $kn+1$ among them.
We put $\alpha_{k}:=\sum_{p\in B}\alpha_{k}(\Ga_{p})$ and call it the {\it $k$-th singularity index} of the fibration.
Let $A$ be the sum of all $(-n)$-curves contained in $\RT$ and put $\RT_0 := \RT - A$. 
Then the $0$-th singularity index is defined by $\alpha_{0}:=(K_{\tilde{\PHI}}+\RT_0)\RT_0$.
The ramification index of $\Ti{R}_{h}$ is defined by $\alpha^{+}_{0}:=(K_{\tilde{\PHI}}+\Ti{R}_{h})\Ti{R}_{h}$,
where  $\Ti{R}_{h}$ denotes the $\Ti{\varphi}$-horizontal part of $\Ti{R}$.
We put
\begin{align*}
\Al_0^{+} (\Ga_{p})&:=    \Ti{R}_h \Ti{\Ga}_{p}-\Sh(\Rm{Supp} (\Ti{R}_{h})\cap \Rm{Supp} (\Ti{\Ga}_{p})), \\
\Al_{0}(\Ga_{p})&:= \Al_{0}^{+}(\Ga_{p})-2\sharp(\Ti{R}_{0})_{v}(p).
\end{align*}
Then we note that $\alpha_{0}=\sum_{p\in B}\alpha_{0}(\Ga_{p})$ and 
$\alpha^{+}_{0}=\sum_{p\in B}\alpha^{+}_{0}(\Ga_{p})$.
For a positive integer $a$, we denote by $j_{a}(\Ga_{p})$ the number of $(-an)$-curves in $\Ti{R}_{v}$,
where $\Ti{R}_{v}$ denotes the $\Ti{\varphi}$-vertical part of $\Ti{R}$.
Put 
\begin{align*}
j _{a} :=\sum_{p \in B} j_{a}(\Ga_{p}),\quad \quad j (\Ga_{p}) :=\sum_{a \geq 1} j_{a}(\Ga_{p}).
\end{align*}

We consider the vertical part $\Ti{R}_v=\sum_{p\in B}\Ti{R}_v(p)$ of $\Ti{R}$ with respect to $\Ti{\PHI}:\Ti{P}\to B$.
To decompose $\Ti{R}_v(p)$, we define a family $\{L_{i}\}_i$ consisting of vertical irreducible curves in $\Ti{R}_v(p)$ as follows.
\begin{itemize}
\item[(i)] 
Let $E_{1}$ be the fiber $\Gamma_{p}$ or a $(-1)$-curve on a ruled surface between $\Ti{P}$ and $P$.
Let $L_1$ be the proper transform of $E_1$ on $\Ti{P}$.

\item[(ii)] For $i \geq 2$, $L_{i}$ is the proper transform of an exceptional $(-1)$-curve $E_i$ that is contracted to a point $x_i$ in $E_k$ or its proper transform for some $k<i$.
\end{itemize}
Put $\Fr{L}:=\{\{L_{i}\}_{i}\mid$ $\{L_{i}\}_{i}$ satisfies (i) and (ii)$\}$. 
If $\{L_{i}\}_{i} $ and $\{L'_{j}\}_{j}$ ($\{L_{i}\}_{i} $,$\{L'_{j}\}_{j} \in \Fr{L}$) have a common curve, one can show the union of them $\{L_{i},L'_{j}\}_{i,j}$ is in $\Fr{L}$.
We define a partially order $\{L_{i}\}_{i} \leq \{L'_{j}\}_{j}$ if $\{L_{i}\}_{i} \subset \{L'_{j}\}_{j}$. 
Then $(\Fr{L},\leq)$ is a partially order set, so there exist maximal elements $\{L_{1,k_1}\}_{k_{1}}, \cdots ,\{L_{\eta_{p},k_{\eta_{p}}}\}_{k_{\eta_{p}}}$, where $\eta_{p}$ is the number 
of maximal elements of $\Fr{L}$.
It's an abuse of the symbol, but we replace indices $k_{1}, \cdots , k_{\eta_{p}}$ to $k$.
We note that each two $\{L_{1,k}\}_{k}, \cdots ,\{L_{\eta_{p},k}\}_{k}$ have a no common curve.
We put $D_{t} := \sum _{k \geq 1} L_{t,k}$ for $t=1, \cdots , \eta_{p}$.
%

We can describe $\Ti{R}_{v}(p)$ by above $D_{t}$'s as the following.
 $\Ti{R}_{v}(p)$ is decomposed  a disjoint sum consisted of such sum uniquely.
 We denote as
\begin{align}
\label{6/16.1} 
\Ti{R}_{v}(p)= D_{1}+ \cdots  +D_{\eta_{p}}  \quad D_{t} =\sum_{k \geq 1}L_{t,k}.
\end{align}
Let $C_{t,k}$ be the exceptional $(-1)$-curve whose proper transform to $\Ti{P}$ is $L_{t,k}$.
We let $\iota^t (\Ga_{p})$ and $\kappa^t (\Ga_{p})$ denote the numbers of singular points (including infinitely near one) of branch locus
of types $n\Z$ and $n\Z +1$, respectively, at which the proper transforms of two curves from $\{C_{t,k}\}_k$ meet, 
and put $\iota(\Ga_p)=\sum_{t=1}^{\eta_{p}}\iota^t(\Ga_p)$ and 
$\kappa(\Ga_p)=\sum_{t=1}^{\eta_{p}}\kappa^t(\Ga_p)$.

\smallskip

Let $\Ha{\varphi}:\Ha{P}\to B$ be any intermediate ruled surface between $\Ti{P}$ and $P$, and regard $\Ti{\psi}:\Ti{P}\to P$ as the composite of 
the natural birational morphisms $\Ha{\psi}: \Ti{P}\to \Ha{P}$ and $\Ch{\psi}:\Ha{P}\to P$ as $\Ti{\psi}=\Ch{\psi}\circ\Ha{\psi}$:
\begin{align}
\label{intermediate}
  \xymatrix{
    \Ti{P} \ar[r]^{\Ha{\psi}} \ar[dr]_{\Ti{\PHI}} & \Ha{P} \ar[d]^{\Ha{\PHI}} \ar[r]^{\Ch{\psi}}& P \ar[ld]^{\PHI}\\
     & B &
}
\end{align}
We put $\Ha{R}=\Ha{\psi}_*\Ti{R}$.
The fiber of $\Ha{\PHI}$ over $p \in B$ will be denoted by $\Ha{\Ga}_{p}$.
Let $\Ha{\Al}_{k}(\Ga_{p})$ and $\Ch{\Al}_{k}(\Ga_{p})$ be the number
which contributes $\Al_{k}(\Gamma_{p})$
appearing in $\Ha{\psi}$ and $\Ch{\psi}$, respectively. 
We note that the number of singular points of $\Ha{R}$ on $\Ha{\Ga}_{p}$ is counted by $\Ha{\Al}_{k}(\Ga_{p})$.
Note that $\alpha_k(\Ga_p)=\Ha{\Al}_{k}(\Ga_{p})+\Ch{\Al}_{k}(\Ga_{p})$. 
We put $\Ha{\Al}_{k}:=\sum_{p \in B} \Ha{\Al}_{k}(\Ga_{p})$ and $\Ch{\Al}_{k}:=\sum_{p \in B} \Ch{\Al}_{k}(\Ga_{p})$.

\smallskip

Choose and fix $p \in B$ and $\Ha{z} \in \Ha{\Ga}_p$.  
We let $\Ti{R}_v(p)_{\Ha{z}}$ be the biggest subdivisor of $\Ti{R}_v(p)$ contracted to $\Ha{z}$ by $\Ha{\psi}$.
Note that we have $\Ti{R}_v(p)_{\Ha{z}}\neq 0$ only when there is a singular point (of the branch locus) whose multiplicity is in $n\mathbb{Z}+1$ (and $>1$) and which is infinitely near to $\Ha{z}$ 
(including $\Ha{z}$ itself).
Note also that $\Ti{R}_v(p)_{\Ha{z}}$ is a disjoint union of non-singular rational curves each of which is a $(-an)$-curve for some positive integer $a$.
Let $\Ti{\Ga}_{p,\Ha{z}}$ be the biggest subdivisor of $\Ti{\Ga}_{p}$ which is contracted to $\Ha{z}$ by $\Ha{\psi}$.
We put
\begin{equation}\label{eq(2.2)}
\begin{aligned}
\Al_0^{+} (\Ga_{p})_{\Ha{z}} &:=
  \begin{cases}
    \Ti{R}_h \Ti{\Ga}_{p,\Ha{z}}-\Sh\left(\Rm{Supp} (\Ti{R}_{h})\cap \Rm{Supp} (\Ti{\Ga}_{p,\Ha{z}})\right) &({\rm if\;}\Ha{R} {\rm\;is\;singular\;at\;} \Ha{z}\;),\\
    \\
    \left({\rm The\;ramification\;index\;of\;}\Ha{\varphi}|_{\Ha{R}_{h}}:\Ha{R}_{h} \to B{\rm \;at\;}\Ha{z}\right)-1 &({\rm if\;}\Ha{R} {\rm\;is\;smooth\;at\;} \Ha{z}\;),\\
    \\
   0 & ({\rm if\;} \Ha{z} \not\in \Ha{R}),
  \end{cases}\\
\Ha{\Al}_{0}(\Ga_{p})_{\Ha{z}}&:= \Al_{0}^{+}(\Ga_{p})_{\Ha{z}}-2\sharp(\Ti{R}_{0})_{v}(p)_{\Ha{z}}.
\end{aligned}
\end{equation}
We note that 
\[
\Al_{0}^{+}(\Ga_{p})=\sum_{\Ha{z} \in \Ha{\Gamma}_{p}}\Al_{0}^{+}(\Ga_{p})_{\Ha{z}}.
\]
For a positive integer $a$, we denote by $j_{a}(\Ga_{p})_{\Ha{z}}$ 
the number of $(-an)$-curves in $\Ti{R}_{v}(p)_{\Ha{z}}$.
Let $\Ch{j}_{1}(\Ga_{p})$ be the number of irreducible curves of $\Ha{R}_{v}(p)$ whose proper transforms 
on $\Ti{P}$ are $(-n)$ curves. 
Then we have
\begin{equation}\label{eq(2.3)}
j_{1}(\Ga_{p})= \sum_{\Ha{z}\in \Ha{\Ga}_{p}}j_{1}(\Ga_{p})_{\Ha{z}}+\Ch{j}_{1}(\Ga_{p}),\quad
\Al_{0}(\Ga_{p})=\sum_{\Ha{z}\in \Ha{\Ga}_{p}}\Ha{\Al}_{0}(\Ga_{p})_{\Ha{z}}-2\left(\sharp\Ha{R}_{v}(p)-\Ch{j}_{1}(\Ga_{p})\right).
\end{equation}

\smallskip

\begin{lem}[\cite{Eno} Lemma~5.2]
\label{lem5.2}
The following hold:

\smallskip

$(1)$  $\iota (\Ga_{p})=j(\Ga_{p})-\eta_{p}$.

\smallskip

$(2)$ $\Al_{0}^{+}(\Ga_{p})\geq (n-2)(j(\Ga_{p})-\eta_{p}+2\kappa (\Ga_{p}))$.

\smallskip

$(3)$ $\displaystyle{\sum_{k \geq 1}\Al_{k}(\Ga_{p})\geq \sum _{a\geq 1}(an-2)j_a (\Ga_{p})+2\eta_{p}-\kappa (\Ga_{p})}$.
\end{lem}

\begin{lem}
\label{10/18.1}
Assume that $n \geq 3$.
Then it hold that 
\begin{align*}
\alpha_{0}(\Ga_{p})+B_{n}\sum_{k \geq 1}\alpha_{k}(\Ga_{p})-nB_{n} \geq 0,
\end{align*}
where $B_{n}=2/n$ if $n \geq 4$ and $B_{3}=1/2$.
\end{lem}

\begin{proof}
From 
$
\Al_{0}(\Ga_{p}) = \Al_{0}^{+}(\Ga_{p})-2\sharp(\Ti{R}_{0})_{v}(p)= \Al_{0}^{+}(\Ga_{p})-2\sum_{a\geq2} j_{a}(\Ga_{p})
$
and Lemma~\ref{lem5.2} (2), 
we have 
\begin{align*}
\Al_{0}(\Ga_{p}) &\geq (n-2)(j(\Ga_{p})-\eta_{p}+2\kappa(\Ga_{p}))
-2\sum_{a\geq2} j_{a}(\Ga_{p})\\
 &=(n-4)\sum_{a\geq2} j_{a}(\Ga_{p})-(n-2)\eta_{p}+
 2(n-2)\kappa(\Ga_{p})+(n-2) j_{1}(\Ga_{p}).
\end{align*}
By Lemma~\ref{lem5.2} (3), 
we have
\begin{align*} 
B_{n} \sum_{k \geq 1}\Al_{k}(\Ga_{p})\geq \sum _{a\geq 1}B_{n}(an-2)j_a (\Ga_{p})
+2B_{n}\eta_{p}-B_{n}\kappa (\Ga_{p}).
\end{align*}
Therefore we get 
\begin{align*}
 &\Al_{0}(\Ga_{p}) + B_{n}\sum_{k \geq 1}\Al_{k}(\Ga_{p})\\
 &\geq \sum _{a\geq 2} \left((n-4)+B_{n}(an-2) \right)j_a (\Ga_{p})
 +\left(2B_{n}-(n-2) \right)\eta_{p}\\
 &+\left(2(n-2) - B_{n}\right)\kappa (\Ga_{p})
 +\left( (n-2) + B_{n}(n-2)\right)  j_1 (\Ga_{p})\\
 &=  \sum _{a\geq 2}  \left(-2+anB_{n} \right) j_a (\Ga_{p})
 +\left( n-2-2B_{n}\right)( j (\Ga_{p})-\eta_{p})\\
 &+\left(2(n-2) - B_{n}\right)\kappa (\Ga_{p})
 +n B_{n} j_1 (\Ga_{p}).
  \end{align*}
 By $B_{n}=2/n$ if $n \geq 4$ and $B_{3}=1/2$,
 the coefficients of $j_a (\Ga_{p})$, $ j (\Ga_{p})-\eta_{p}=\iota (\Ga_{p})$ 
 and $\kappa (\Ga_{p})$ are all non-negative.
 \end{proof}
We discuss the localization of the invariant $K_{f}^2$.
We recall the following lemma.

\begin{lem}[\cite{Eno}]
\label{lem1.7}
Let $f:S\to B$ be a primitive cyclic covering fibration of type $(g,0,n)$.
Then it holds 
\begin{align*}
K_{f}^2=\frac{n-1}{r-1}\biggl(\frac{(n-1)r-2n}{n}(\Al_{0}-2j_{1})+(n+1)\sum_{k\geq 1}k(r-nk)\Al_{k} \biggr)
-n\sum_{k\geq 1}\Al_{k} + j_{1}.
\end{align*}
\end{lem}

So if we put
\begin{equation}
\label{localinv}
\begin{aligned}
K_{f}^{2}(\Ga_{p})&:=\frac{n-1}{r-1}\biggl(\frac{(n-1)r-2n}{n}(\Al_{0}(\Ga_p)-2j_{1}(\Ga_p))+(n+1)\sum_{k\geq 1}k(r-nk)\Al_{k}(\Ga_p) \biggr)\\
&-n\sum_{k\geq 1}\Al_{k}(\Ga_p) + j_{1}(\Ga_p)
\end{aligned}
\end{equation}
for $p \in B$, then we get $K_{f}^2 = \sum_{p \in B}K_{f}^{2}(\Ga_{p})$.
Since we have $K_{f}^{2}(\Ga_{p})=0$ except for a finite number of points $p\in B$, we see that the invariant $K_f^2$ has been localized 
to a finite number of fibers.
We will further localize a part of $K_{f}^{2}(\Ga_{p})$ to points on the fiber $\Ha{\Ga}_{p}$ in the diagram (\ref{intermediate}). 

For a fiber $\Ha{\Ga}_p$ and a point $\Ha{z} \in \Ha{\Ga}_{p}$, we put 
\begin{equation}
\label{localinv2}
\begin{aligned}
K_{f}^{2}(\Ga_{p})_{\Ha{z}}:=&\frac{n-1}{r-1}\frac{(n-1)r-2n}{n}\left(\Ha{\Al}_{0}(\Ga_{p})_{\Ha{z}}-2j_{1}(\Ga_{p})_{\Ha{z}}\right)+j_{1}(\Ga_{p})_{\Ha{z}}\\
+&\frac{1}{r-1}\sum_{k\geq 1}\bigl(  (n^2-1)k(r-nk)-(r-1)n\bigr)\Ha{\Al}_{k}(\Ga_{p})_{\Ha{z}}. 
\end{aligned}
\end{equation}
Then we get the following:

\begin{prop}
\label{loc'nofK}
Let $f:S\to B$ be a primitive cyclic covering fibration of type $(g,0,n)$.
Let $\Ti{\psi}=\Ch{\psi}\circ\Ha{\psi}$ be an arbitrary decomposition of $\Ti{\psi}:\Ti{P} \to P$ to have the commutative diagram
\[
  \xymatrix{
    \Ti{P} \ar[r]^{\Ha{\psi}} \ar[dr]_{\Ti{\PHI}} & \Ha{P} \ar[d]^{\Ha{\PHI}} \ar[r]^{\Ch{\psi}}& P \ar[ld]^{\PHI}\\
     & B. &
}
\]
Denote a fiber of $\Ha{\PHI}$ over $p \in B$ by $\Ha{\Ga}_{p}$.
Then it holds 
\begin{align*}
K_{f}^{2}(\Ga_{p})&=\sum_{\Ha{z} \in \Ha{\Ga}_{p}} K_{f}^{2}(\Ga_{p})_{\Ha{z}}
 -2\frac{n-1}{r-1}\frac{(n-1)r-2n}{n}\sharp\Ha{R}_{v}(p)+\Ch{j}_{1}(\Ga_{p})\\
 &+\frac{1}{r-1}\sum_{k\geq 1}\bigl(  (n^2-1)k(r-nk)-(r-1)n\bigr)\Ch{\Al}_{k}(\Ga_{p}).
\end{align*}
\end{prop}
\begin{proof}
By (\ref{localinv}), (\ref{eq(2.3)}), we get 
\begin{align*}
K_{f}^{2}(\Ga_{p})=&\frac{n-1}{r-1}\frac{(n-1)r-2n}{n}\left(\Al_{0}(\Ga_p)-2j_{1}(\Ga_p)\right)+ j_{1}(\Ga_p)\\
 &+\frac{1}{r-1}\sum_{k \geq 1} \left( (n^2-1)k(r-nk)-(r-1)n \right)\Al_{k}(\Ga_{p})\\
 =&\frac{n-1}{r-1}\frac{(n-1)r-2n}{n}\left(\sum_{\Ha{z}\in \Ha{\Ga}_{p}}\Ha{\Al}_{0}(\Ga_{p})_{\Ha{z}}-2\left(\sharp\Ha{R}_{v}(p)-\Ch{j}_{1}(\Ga_{p})\right)
-2\sum_{\Ha{z}\in \Ha{\Ga}_{p}} j_1 (\Ga_{p})_{\Ha{z}} -2\Ch{j}_{1}(\Ga_p)\right)\\
 &+\sum_{\Ha{z}\in \Ha{\Ga}_{p}} j_1 (\Ga_{p})_{\Ha{z}} +\Ch{j}_{1}(\Ga_p)\\
 &+\frac{1}{r-1}\sum_{k \geq 1} \left( (n^2-1)k(r-nk)-(r-1)n \right)\left(\sum_{\Ha{z}\in \Ha{\Ga}_{p}}\Ha{\Al}_{k}(\Ga_{p})_{\Ha{z}}+\Ch{\Al}_{k} (\Ga_{p})\right)\\
%
 %
 =&\sum_{\Ha{z} \in \Ha{\Ga}_{p}} K_{f}^{2}(\Ga_{p})_{\Ha{z}}
 -2\frac{n-1}{r-1}\frac{(n-1)r-2n}{n}\sharp\Ha{R}_{v}(p)+\Ch{j}_{1}(\Ga_{p})\\
 &+\frac{1}{r-1}\sum_{k\geq 1}\bigl(  (n^2-1)k(r-nk)-(r-1)n\bigr)\Ch{\Al}_{k}(\Ga_{p})
\end{align*}
as wished.
\end{proof}

\begin{lem}
\label{lowerbdofK}
Assume that $n \geq 3$.
Then the following hold:
\begin{align*}
 (r-1)K_{f}^{2}(\Ga_{p})\geq & \sum_{k\geq 1}\biggl( -\frac{n-1}{n}\bigr((n-1)r-2n \bigl)B_{n}
+(n^2 -1)k(r-nk)-n(r-1) \biggr)\Al_{k}(\Ga_{p})
\end{align*}
In particular, $K_{f}^{2}(\Ga_{p})$ is non-negative.
\end{lem}

\begin{proof}
From  Lemma~\ref{10/18.1}, we have
\begin{align*}
\Al_{0}(\Ga_{p})-nB_{n}j_{1}(\Ga_{p})\geq -B_{n}\sum_{k \geq 1}\Al_{k}(\Ga_{p}).
\end{align*}
Then plugging it to (\ref{localinv}), we obtain the desired inequality 
\begin{align*}
 (r-1)K_{f}^{2}(\Ga_{p})\geq & \sum_{k\geq 1}\biggl( -\frac{n-1}{n}\bigr((n-1)r-2n \bigl)B_{n}
+(n^2 -1)k(r-nk)-n(r-1) \biggr)\alpha_{k}(\Ga_{p}).
\end{align*}
\end{proof}

\begin{lem}
\label{10/18.2}
Let the notation and the assumption as above.
Then it holds that  
\begin{align*}
 (r-1)K_{f}^{2}(\Ga_{p})_{\Ha{z}} 
& \geq   \frac{n-1}{n}\left((n-1)r-2n \right){\Al}_{0}^{+}(\Gamma_{p})_{\Ha{z}}+(r-1)j_{1}(\Gamma_{p})_{\Ha{z}}  \\
& + \sum_{k\geq 1}\biggl( -2\frac{n-1}{n}\bigr((n-1)r-2n \bigl)
+(n^2 -1)k(r-nk)-n(r-1) \biggr)\Ha{\Al}_{k}(\Ga_{p})_{\Ha{z}} \\
  \end{align*}
\end{lem}
\begin{proof}
By the definition of $K_{f}^{2}(\Ga_{p})_{\Ha{z}}$, we have 
\begin{align*}
(r-1)K_{f}^{2}(\Ga_{p})_{\Ha{z}}=&\frac{n-1}{n}\left((n-1)r-2n\right)\Al_{0}^{+}(\Ga_{p})_{\Ha{z}}
-2\frac{n-1}{n}\left((n-1)r-2n\right)j(\Ga_{p})_{\Ha{z}}\\
+&(r-1)j_{1}(\Ga_{p})_{\Ha{z}}+\sum_{k\geq 1}\bigl(  (n^2-1)k(r-nk)-(r-1)n\bigr)\Ha{\Al}_{k}(\Ga_{p})_{\Ha{z}}. 
\end{align*}
Let $j^{k}(\Gamma_{p})_{\Ha{z}}$ be the number of an irreducible component contributing 
$j(\Gamma_{p})_{\Ha{z}}$ which arises from a singular point of multiplicity $nk+1$.
Then we have $j(\Gamma_{p})_{\Ha{z}} =\sum_{k \geq 1}j^{k}(\Gamma_{p})_{\Ha{z}}$.
Since $\Ha{\Al}_{k}(\Ga_{p})_{\Ha{z}} \geq j^{k}(\Gamma_{p})_{\Ha{z}}$, we have 
\begin{align*}
 (r-1)K_{f}^{2}(\Ga_{p})_{\Ha{z}} 
 &\geq \frac{n-1}{n}\left((n-1)r-2n \right)\Al_{0}^{+}(\Gamma_{p})_{\Ha{z}}+(r-1)j_{1}(\Gamma_{p})_{\Ha{z}}  \\
& + \sum_{k\geq 1}\biggl( -2\frac{n-1}{n}\bigr((n-1)r-2n \bigl)
+(n^2 -1)k(r-nk)-n(r-1) \biggr)\Ha{\Al}_{k}(\Ga_{p})_{\Ha{z}} . 
  \end{align*}
\end{proof}

\begin{lem}
\label{addlem}
Let the notation and the assumption be as in the above.
Then it holds that
\begin{align*}
 (r-1)K_{f}^{2}(\Ga_{p})
 &\geq  \sum_{\Ha{z}\in \Ha{\Gamma}_{p}}  \frac{n-1}{n}\left((n-1)r-2n \right)\Al_{0}^{+}(\Gamma_{p})_{\Ha{z}}+\sum_{\Ha{z}\in \Ha{\Gamma}_{p}}(r-1)j_{1}(\Gamma_{p})_{\Ha{z}}  \\
 & + \sum_{\Ha{z}\in \Ha{\Gamma}_{p}} \biggl\{ \sum_{k\geq 1}\biggl( -2\frac{n-1}{n}\bigr((n-1)r-2n \bigl)
+(n^2 -1)k(r-nk)-n(r-1) \biggr) \biggr\}\Ha{\Al}_{k}(\Ga_{p})_{\Ha{z}} \\
 &-2\frac{n-1}{n}\left((n-1)r-2n\right)\sharp\Ha{R}_{v}(p)+(r-1)\Ch{j}_{1}(\Ga_{p})\\
 &+ \sum_{k\geq 1}\left((n^2 -1)k(r-nk)-n(r-1) \right)\Ch{\alpha}_{k}(\Ga_{p}).\\
 \end{align*}
\end{lem}
\begin{proof}
By the Proposition~\ref{loc'nofK},
we have
 \begin{align*}
(r-1)K_{f}^{2}(\Ga_{p})&=\sum_{\Ha{z} \in \Ha{\Ga}_{p}} K_{f}^{2}(\Ga_{p})_{\Ha{z}}
 -2\frac{n-1}{n}\left((n-1)r-2n\right)\sharp\Ha{R}_{v}(p)+(r-1)\Ch{j}_{1}(\Ga_{p})\\
 &+\sum_{k\geq 1}\bigl(  (n^2-1)k(r-nk)-(r-1)n\bigr)\Ch{\Al}_{k}(\Ga_{p}).
\end{align*}
Plugging the inequality of Lemma~\ref{10/18.2} to the above equality,
we obtain the desired inequality.
\end{proof}

\section{Automorphism of a fibered surface}

It is classically known that finite subgroups of $\Aut(\PROJ^1)$ can be classified as in the following table:
\begin{table}[htb]
  \begin{tabular}{|ll|c|c|} \hline
  & & Order & Length of non-trivial orbits \\ \hline 
  $\Z_{l}$ &Cyclic group &$ l $  & $1$, $1$ \\
  $D_{2l}$ &Dihedral group  & $2l$ ($l\geq 2$) & $l$, $l$, $2$ \\
   $T_{12}$ &Tetrahedral group  &$12$  & $4$, $4$, $6$ \\ 
   $O_{24}$ &Octahedral group  & $24$  &$ 6$, $8$, $12$ \\
 $I_{60}$ &Icosahedral group  & $60$  & $12$, $20$, $30$ \\ \hline
  \end{tabular}
  \caption{Finite subgroups of $\Rm{Aut}(\Bb{P}^{1})$}
  \label{subgroup}
\end{table}

In the Table~\ref{subgroup}, the entry ``4, 4, 6'' for example means 
the action of the tetrahedral group has one orbit of length $6$, two orbits of length $4$ and the other orbits are of length $12$.  
Note that the group which has a fixed point is necessarily a cyclic group.

For a fibration $f:S \to B$, we define the automorphism group of $f$ as 
\[
\Aut (f):=\{(\kappa_{S},\kappa_{B})\in \Aut (S)\times \Aut(B) \mid f \circ \kappa_{S} =  \kappa_{B} \circ  f \}.
\]

Let $f:S\to B$ be a primitive cyclic covering fibration of type $(g, 0, n)$.
Put $\Rm{Aut}(S/B):=\{(\kappa_{S}, \Rm{id})\in \Rm{Aut}(f)\}$.
We can regard $\Sigma$ as subgroup of $\Rm{Aut}(S/B)$,
where $\Sigma$ is defined in Remark~\ref{basic}.
Take an arbitrary finite subgroup $G$  of $\Aut (f)$.
Since $\Aut (S/B)$ is a finite group, we may assume $(\sigma,\Rm{id}) \in G$ to estimate the order of it.
We have the exact sequence 
\begin{align*}
1\to K \to G \to H \to1, 
\end{align*}
where 
$K:=\{(\kappa_{S}, \Rm{id})\in G\}$ and  $H:=\{\kappa_{B} \in \Aut (B)\mid  (\kappa_{S},\kappa_{B})\in G, \exists \kappa_{S}\in \Aut (S)\}$.
$K$ can be considered as a subgroup of $\Aut(F)$ for a general fiber $F$ of $f$.

\begin{prop}
\label{homlemma}
Let $\theta_{1}:F_1 \to \Gamma_{1}$ and $\theta_{2}:F_2 \to \Gamma_{2}$ be  totally ramified coverings of degree $n$ between smooth projective curves.
Let $r$ be a number of ramification points of $\theta_{1}$.
Assume $r>  2(\frac{n^{2}}{n-1}+2g(\Gamma_{2})-2)$.
Then, for any  isomorphism $\kappa_{F}:F_{1} \to F_{2}$ 
 there exists an automorphism $\kappa_{\Gamma}:\Gamma_{1} \to \Gamma_{2}$ 
which fit into the commutative diagram
\begin{equation*}
 \vcenter{ \xymatrix{ 
 F_{1} \ar[r]^{\kappa_{F}} \ar[d]_{\theta_{1}} 
 & F_{2} \ar[d]^{\theta_{2}} \\
 \Gamma_{1} \ar[r]_{\kappa_{\Gamma}} 
 & \Gamma_{2}} }
 \end{equation*}
\end{prop}

\begin{proof}
We have a commutative diagram
\[\xymatrix{
F_{1} \ar[rd]^{\Psi} \ar@/_1.5pc/[ddr]_{\theta_{1}} \ar@/^1.5pc/[rrd]^{\theta_{2}\circ \kappa_{F}}&  & \\
  & \Gamma_{1}\times  \Gamma_{2}  \ar[r]^{p_2} \ar[d]_{p_1}& \Gamma_{2}\\
& \Gamma_{1} &
}\]
where $p_{i} : \Gamma_{1} \times \Gamma_{2} \to \Gamma_{i}$ is the projection for $i=1, 2$ and 
$\Psi:F_{1} \to \Gamma_{1} \times \Gamma_{2}$ is induced by the universality of the fibre product.
Let $d:=\Rm{deg}\;( \Psi:F_1 \to \Psi(F_{1}))$.
If $d=n$, then $p_{1}|_{\Psi(F_{1})}:\Psi(F_{1})\to \Gamma_{1}$ is an isomorphism by Zariski's main theorem.
Hence we obtain the desired isomorphism $\kappa_{\Gamma} =p_{2}|_{\Psi(F_1)}\circ (p_{1}|_{\Psi(F_{1})})^{-1}$. 
So we may assume $d < n$.
Then the degree of $p_{1}|_{\Psi(F_1)}$ and $p_{2}|_{\Psi(F_1)}$ are $n/d$.
We will show $r \leq 2(\frac{n^{2}}{n-1}+2g(\Gamma_{2})-2)$.
Let $\nu:E \to \Psi(F_1)$ be the normalization of $\Psi(F_1)$.
It is well known that $p_{a}(\Psi(F_1))\geq g(E)$.
Since the morphism $p_{1}|_{\Psi(F_1)}\circ \nu$ is a totally ramified cover, we have
\begin {align}
\label{geometricgenus}
2g(E)-2=\frac{n}{d}(2g(\Gamma_{1})-2)+(\frac{n}{d}-1)r
\end {align}  
by the Hurwitz formula.

On the other hand, we have 
\begin{align}
\label{12/21,1}
2p_{a}(\Psi(F_1))-2&=(K_{\Gamma_{1}\times \Gamma_{2}}+\Psi(F_1))\Psi(F_1)\\
\nonumber
 &=(p_{1}^{\ast}K_{\Gamma_{1}}+p_{2}^{\ast}K_{\Gamma_{2}}+\Psi(F_1))\Psi(F_1)\\
\nonumber
 &=\frac{n}{d}(2g(\Gamma_{1})-2)+\frac{n}{d}(2g(\Gamma_{2})-2)+(\Psi(F_1))^2.
\end{align}
Let $I_{i}$ be a fiber of $p_{i}$ for $i=1,2$.
Noting that $(I_{1}+I_{2})^2=2I_{1} I_{2}=2>0$,
we have
\begin{align}
\label{12/21,2}
 \left(\frac{2n}{d}\right)^{2} = (\Psi(F_1)(I_{1}+I_{2}))^2 \geq 2\Psi(F_1)^2
\end{align} 
by the Hodge index theorem.
Therefore we have 
\begin{align*}
2(\frac{n}{d})^2 \geq  2p_{a}(\Psi(F_1))-2-\frac{n}{d}(2g(\Gamma_{1})-2)-\frac{n}{d}(2g(\Gamma_{2})-2).
\end{align*}
from (\ref{12/21,1}) and (\ref{12/21,2}), which is equivalent to 
\begin{align*}
2\left(\frac{n}{d}\right)^2+\frac{n}{d}(2g(\Gamma_{1})-2)+\frac{n}{d}(2g(\Gamma_{2})-2)\geq 2p_{a}(\Psi(F_1))-2.
\end{align*}
From (\ref{geometricgenus}), we have 
\begin{align*}
2\left(\frac{n}{d}\right)^2+\frac{n}{d}(2g(\Gamma_{1})-2)+\frac{n}{d}(2g(\Gamma_{2})-2)\geq \frac{n}{d}(2g(\Gamma_{1})-2)+\left(\frac{n}{d}-1\right)r.
\end{align*}
It implies that $r \leq 2(\frac{n^{2}}{n-1}+2g(\Gamma_{2})-2)$.
\end{proof}

\begin{lem}
\label{orbit-lem}
Assume that $r \geq 3n$. 
Take a smooth fiber $F$ of $f$ and arbtrary point $z \in F$.
For any $(\kappa_{S}, \kappa_{B}) \in G$, it holds
\begin{align*}
\kappa_{S}(\Sigma\cdot z)=\Sigma\cdot{\kappa_{S}(z)},
\end{align*} 
where $\Sigma\cdot z$ denotes the $\Sigma$-orbit of $z$.
\end{lem}

\begin{proof}
Actions $\sigma|_{F}$ and $\sigma|_{\kappa_{S}(F)}$ 
induce quotient morphisms $\theta_{1}:F \to \Bb{P}^1$ and  $\theta_{2}:\kappa_{S}(F) \to \Bb{P}^1$ of degree $n$, respectively.
From the assumption that $r\geq 3n$, there exists a isomorphism $\kappa_{\Bb{P}}:\Bb{P}^1 \to \Bb{P}^1$ such that the diagram 
\begin{equation*}
 \vcenter{ \xymatrix{ 
 F \ar[r]^{\kappa_{S}} \ar[d]_{\theta_{1}} 
 & \kappa_{S}(F) \ar[d]^{\theta_{2}} \\
 \Bb{P}^1 \ar[r]_{\kappa_{\Bb{P}}} 
 & \Bb{P}^1} } 
 \end{equation*}
commutes for any $\kappa_{S} \in G$ by Proposition \ref{homlemma}.
Hence we obtain $\kappa_{S}(\Sigma\cdot{z})=\Sigma\cdot{\kappa_{S}(z)}$.
\end{proof}

\medskip

In what follows, we tacitly assume that $r\geq 3n$.

\medskip

\begin{lem}
\label{normalgp-Lem1}
The covering transformation group $\Sigma$ is a normal subgroup of $G$.
\end{lem}

\begin{proof}
Let $\Rm{Fix}^{1}(\Sigma)$ be the closed subvariety which consists with $f$-horizontal
$1$-dimensional fixed locus of $\Sigma$.
We note that $\Rm{Fix}^{1}(\Sigma)$ is a finite disjoint union of smooth curves.
From Lemma~~\ref{orbit-lem},
we have 
$
\{\kappa_{S}(z)\}=\kappa_{S}(\Sigma\cdot z)=\Sigma  \cdot {\kappa_{S}(z)}
$
for any $(\kappa_{S}, \kappa_{B}) \in G$ and general $z\in\Fix^{1}(\Sigma)$.
Hence there exists a group homomorphism 
\begin{align*}
\lambda:G \to \Rm{Aut}(\Rm{Fix^{1}{\Sigma}}),
\end{align*}
where $\Rm{Aut}(\Rm{Fix^{1}{\Sigma}})$ is the automorphism group of $\Rm{Fix^{1}{\Sigma}}$.
It suffices for our purpose to show that $\Ker\;{\lambda}=\Sigma$.
It is clear that $\Sigma\subset \Ker\;\lambda$, we shall show $\Ker\;\lambda \subset \Sigma$.
Since $(\kappa_{S}, \kappa_{B}) \in \Ker\;\lambda$ fixes  $\Rm{Fix^{1}{\Sigma}}$, 
we have $\kappa_{B}=\Rm{id}$.
Let $F$ be a smooth fiber of $f$ and let $\theta:F \to F/\Sigma = \Bb{P}^{1}$.
There exists a isomorphism $\kappa_{\Bb{P}}:\Bb{P}^1 \to \Bb{P}^1$ such that the diagram 
\begin{equation*}
 \vcenter{ \xymatrix{ 
 F \ar[r]^{\kappa_{S}|_{F}} \ar[d]_{\theta} 
 & F \ar[d]^{\theta} \\
 \Bb{P}^1 \ar[r]_{\kappa_{\Bb{P}}} 
 & \Bb{P}^1} } 
 \end{equation*}
commutes for any $\kappa \in \Rm{Ker}\;\lambda$ by Proposition \ref{homlemma}. 
Since the action $\kappa|_{F}$ fixes every point in $\Fix^{1}(\Sigma)\cap F$, 
which consisted with at least $r$ points,
we see that $\kappa_{\Bb{P}}$ fixes more than two points.
So $\kappa_{\Bb{P}}$ is the identity, i.e., $\kappa_{S}|_{F} \in \Aut(F/\PROJ^1)$. 
Since $F$ is a general fiber of $f$, we have $(\kappa_{S},\Rm{id}) \in \Sigma$.
\end{proof}


From Lemma~{\ref{normalgp-Lem1}}, the set of isolated fixed point $\Sigma$ is $G$-stable.
Since $\Ti{S} \to S$ is a minimal succession of blowing-ups that resolves all isolated fixed point of $\Sigma$,
the action of $G$ on $S$ can be lifted to the one on $\Ti{S}$ and we can regard $G$ as subgroup of $\Rm{Aut}(\Ti{f})$.
We put $\Ti{G}=G/\Sigma$ and $\Ti{K}=K/\Sigma$, then we have three exact sequence  
\begin{align*}
 &1\to \Sigma \to G \to \Ti{G} \to 1,\\
 &1\to \Sigma \to K \to \Ti{K} \to 1,\\
 &1\to \Ti{K} \to \Ti{G} \to H \to1.
\end{align*}
Note that $\Ti{G} \subset \Aut (\Ti{\varphi})$, $\Ti{K}\subset \Rm{Aut}(\Ti{P}/B)$ and $\Ti{R}$ is $\Ti{G}$-stable (namely $\Ti{G}(\Ti{R})=\Ti{R}$).
\smallskip

\begin{lem}
\label{hyoujyunka}
There is a relatively minimal model $\PHI: P \to B$ of $\widetilde{\PHI}$
satisfying the following two conditions.

\begin{itemize}
\item $\Mult_{x} R_{h} \leq r/2$ holds for any point $x$ of the $\PHI$-horizontal part $R_h$ of $R$.

\item There exists a finite (possibly empty) subset $\Delta$ of $B$ such that the action of $\KT$ descends down faithfully on 
$P\setminus \PHI^{-1}(\Delta)$ but not over $\Delta$. One can find one point on each fiber of $\PHI$ in $\PHI^{-1}(\Delta)$ and, 
if $\overline{P}$ denotes the blowing-up at such $\sharp \Delta$ points on $P$, then the action of $\KT$ descends 
down faithfully on $\overline{P}$ which can switch two irreducible components on each singular fiber of the natural ruling 
$\overline{\varphi}:\overline{P}\to B$.
\end{itemize}
\end{lem}

\begin{proof}
Let $\GAT_{p}$ be an arbitrary singular fiber of $\PT\to B$, 
and $E$ be any $(-1)$-curve contained in $\GAT_{p}$.
Then its $\KT$-orbit $\KT\cdot E$ satisfies one of the following:
\begin{itemize}
  \item  $\KT\cdot{E}$ consists of disjoint union of $(-1)$-curves in $\GAT_{p}$,
  \item We can find two curves in $\KT\cdot E=\{E_{1}, E_{2}, \cdots, E_{t}\}$ meeting at a point,
  where $t$ is a number of irreducible components of $\KT\cdot E$
\end{itemize}

In the former case, contracting $\KT\cdot{E}$ to points, the action of $\KT$ descends down to the action on the new fiber obtained by the contraction.

In the latter case, we may assume that $E_1$ and $E_2$ meet. 
Since the intersection form on $E_1\cup E_2$ is negative semi-definite, $E_1^2=E_2^2=-1$ and $E_1E_2>0$, 
the only possibility is: $E_1E_2=1$ and $(E_1+E_2)^2=0$.
Since any ruled surface has no multiple fibers, we get $t=2$
and
\begin{align*}
\GAT_{p}=E_{1}+E_{2}, \quad E_{1}E_{2}=1 
\end{align*}
by Zariski's Lemma.
If the action of $\KT$ never switch $E_{1}$ and $E_{2}$,
$\KT$ induces the action on the new fiber 
 obtained by contracting either $E_{1}$ or $E_{2}$.
Otherwise, we stop.

Hence, repeating the above procedure,
we finally get a fiber without $(-1)$-curves or a fiber of type $E_{1}+E_{2}$ (and $\KT$ switches $E_{1}$ and $E_{2}$).

Let $\PHI' :P' \to B$ be the model at which the above algorithm terminates for all singular fibers of $\PHIT$.
\[
\xymatrix{
\PT \ar[r] \ar[rd]_{\PHIT} & P'\ar[d]^{\PHI'}  & \\ 
& B &
}
\]
Denote the image of $\RT$ to $P'$ by $R'$.
Note that any singular fibers of $\PHI'$ are of the form $E_{1}+E_{2}$.

Let $\Ga'$ be a smooth fiber of $P'\to B$. 
We will show that $\Mult _{x} R'_{h}\leq r/2$ can be assumed to hold for all singular points $x$ of $R'_h$ on $\Ga'$. 
When $\KT$ is not cyclic, the length of $\KT$-orbits is not less than 2.
So there are at least two singular points of $ R'_{h}$ on $\Ga'$
which are analytically equivalent.
Therefore the multiplicity of the singular point is not bigger than $r/2$.
Assume that $\KT$ is cyclic. 
If $R'_{h}$ has a singular point on a $\KT$-orbit of length $\Sh\KT$, then 
there are at least two singular points of $R'_{h}$ on $\Ga'$ and we can argue as in the previous case.
So we assume that the singular point of $R'_{h}$ of multiplicity greater than $r/2$ is fixed by $\KT$.
We perform the elementally transformation at the fixed point.
Then $\KT$ acts naturally on the new fiber and one can check easily that the multiplicity of the new branch locus at the resulting fixed point is less than $r/2$ (cf. Lemma~~3.1 of \cite{Eno}).

Let $\PHIB:\PB \to B$ be the model obtained from $\PHI' :P' \to B$ as above, and let $\RB$ be the image of $\RT$ to $\PB$.
Note that the multiplicity of any singular point of $\RB_{h}$ on smooth fibers is at most $r/2$. 
Recall that any singular fiber of $\PHIB$ is of the form $E_{1}+E_{2}$, and $\KT$ can switch 
$E_1$ and $E_2$.
Hence the (singular) points of $\RB_h$ on $E_1$ and on $E_2$ which are switched by $\KT$ are analytically equivalent.
In particular, $\overline{R}_hE_2$ has to coincide with $\overline{R}_hE_1$. 
Then, from $\overline{R}_h(E_1+E_2)=r$, we get $\overline{R}_hE_1=\overline{R}_hE_2=r/2$.
Therefore, we get the desired model $\PHI: P \to B$ after contracting one of $E_{1}$ and $E_{2}$.
\end{proof}


In what follows, we consider the model $\PHIB:\PB \to B$ introduced in the above lemma. 
We denote by $\KB$ the subgroup of $\Aut(\PB/B)$ induced by $\KT$.
We remark the following which has been shown in the proof of the above lemma.

\begin{lem}
\label{badhyoujyunnka}
Assume that $\Delta\neq \emptyset$ and let $\Ga_{p}$ be a fiber over $p \in \Delta $.
Then there exists a singular point of $R_{h}$ of multiplicity $r/2$ on $\Ga_{p}$.
\end{lem}

\begin{lem}
\label{dihedral}
Assume $\Delta\neq \emptyset$. 
Then $\KB$ is either dihedral or $\KB\cong \Z_2$.
\end{lem}
\begin{proof}
Let $\Ov{\Gamma}_{p}$ be the fiber of $\Ov{\varphi}$ over $p \in \Delta $.
On the model $\overline{P}$, $\GAB_{p}$ consists of two $(-1)$-curves $\Ov{\Gamma}_{i}(i=1, 2)$ with $\Ov{\Gamma}_1\Ov{\Gamma}_2=1$ 
 and there exists an element $\kappa_{sw} \in \KB$ such that 
$\kappa_{sw}(\Ov{\Ga}_{1})=\Ov{\Ga}_{2}$, $\kappa_{sw}(\Ov{\Ga}_{2})=\Ov{\Ga}_{1}$
from Lemma~\ref{hyoujyunka}. 

We consider the subgroup 
\begin{align*}
N:=\{\kappa \in \KB \mid \kappa(\Ov{\Ga}_{i}) = \Ov{\Ga}_{i}\quad (i=1, 2)\}.
\end{align*}
For any $\kappa' \in \KB$, we have $\kappa_{sw} \kappa'\in N$ if $\kappa' \not\in N$.
Hence, $\KB=N\cup \kappa_{sw} N$. It follows that $N$ is a normal subgroup of $\KB$ with the quotient group 
$\KB/N\cong \Z_2$.
Since $N$ acts on $\Ov{\Ga}_{i}\cong \PROJ ^1$ and fixes $z_0\in \Ov{\Ga}_1 \cap \Ov{\Ga}_{2}$, we see that $N$ is a cyclic group.
Hence $\KB$ is cyclic or dihedral.

We claim that $\KB$ is either dihedral or $\KB\cong \Z_2$.
We assume that $\KB$ is cyclic and $N$ is not the unit group, and show that this eventually leads us to a contradiction.
Let $\CB_{\KB}$ be the fixed curve of $\KB$ on $\PB$, i.e., the curve traced out by fixed points of $\KB$ on fibers.
Since $\kappa_{sw}$ fixes $\CB_{\KB}$ and the fixed point of $\kappa_{sw}$ on $\Ov{\Ga}_{p}$ 
is only $z_0$, we see that $\CB_{\KB}$ passes through $z_0$.
Since $\CB_{\KB}\Ov{\Ga}_{p}=2$ and $\kappa_{sw}(\Ov{\Ga}_1)=\Ov{\Ga}_2$, we have $\CB_{\KB}\Ov{\Ga}_i=1$ for $i=1,2$.
In particular, $\CB_{\KB}$ is smooth at $z_0$.
Since $N$ is not the unit group, the fixed curve $\CB_{N}$ of $N$ must coincides with $\CB_{\KB}$.
We blow $\overline{P}$ down to $P$. Then the action of $N$ descends down to $P$ locally around $\Ga_{p}$ and the fixed curve on $N$ on $P$ is tangent to $\Ga_{p}$.
This means, however, that $N$ acts trivially on $P$, contradicting that it is not the unit group. 
Hence $\KB$ is dihedral.

\end{proof}

\begin{lem}
\label{Rhassing}
Assume $\Delta\neq \emptyset$. 
Then $\sharp \Delta$ is even.
Furthermore, if $\Ov{K}$ is dihedral, one of the followings holds:
\begin{itemize}
\item[(1)] The branch locus $\Ov{R}$ has a singular point.
\item[(2)] The morphism $\Ov{\varphi}|_{\Ov{R}_{h}}:\Ov{R}_{h} \to B$ has a ramification point.
\end{itemize}
\end{lem}

\begin{proof}
Let $\Ov{\Gamma}_{p}$ be the fiber of $\Ov{\varphi}$ over $p \in \Delta $.
On the model $\overline{P}$, $\GAB_{p}$ consists of two $(-1)$-curves $\Ov{\Gamma}_{i}(i=1, 2)$ with $\Ov{\Gamma}_1\Ov{\Gamma}_2=1$ 
 and there is an element $\kappa_{sw} \in \KB$ such that 
$\kappa_{sw}(\Ga_{1})=\Ga_{2}$, $\kappa_{sw}(\Ga_{2})=\Ga_{1}$
from Lemma~\ref{hyoujyunka}. 
We consider the subgroup 
\begin{align*}
N:=\{\kappa \in \KB \mid \kappa(\Ga_{i}) = \Ga_{i}\quad (i=1, 2)\}.
\end{align*}
Let $\pi: \Ov{P} \to P^{\dagger}:= \Ov{P}/\langle\kappa_{\Rm{sw}}\rangle$ be the quotient map by the cyclic subgroup 
$\langle \kappa_{\Rm{sw}}\rangle  \subset \Ov{K}$ and $\varphi^{\dagger}:P^{\dagger} \to B$ be the natural ruling.
Note that $P^{\dagger}$ is a smooth projective surface.
Since $\kappa_{\Rm{sw}}$ switches irreducible components of an arbitrary fiber over $\Delta$,
 $\varphi^{\dagger}:P^{\dagger} \to B$ is a relatively minimal fibration.
 We denote by $C_{0}^{\dagger}$ and $\Gamma^{\dagger}$ a minimal section with $(C_{0}^{\dagger})^{2}=-e$ and a fiber of $\varphi^{\dagger}$, respectively.
Let $B^{\dagger} \subset P^{\dagger}$ be a branch locus of the double covering $\pi$.
We will show $B^{\dagger} \equiv 2C_{0}^{\dagger} +2e\Gamma^{\dagger}$,
where the symbol $\equiv$ means the numerical equivalence.
Put $B^{\dagger} \equiv 2C_{0}^{\dagger} +b\Gamma^{\dagger}$ with an integer $b$.
We have 
\[
(K_{\varphi^{\dagger}}+B^{\dagger})B^{\dagger}=2(b-e)=\sharp \Delta >0.
\]
Hence we have $b>e$ and $\sharp\Delta$ is even.

We assume $\Ov{K}$ is dihedral. 
Then $N$ is not the trivial group. 
Let  $\Ov{C}_{N}\subset  \Ov{P}$ be the curve fixed by $N$ and $C_{N}^{\dagger}:=\pi_{\ast}\Ov{C}_{N}$.
Put $C_{N}^{\dagger} \equiv C_{0}^{\dagger}+c\Gamma^{\dagger}$ with a non-negative integer $c$.
Since $\Ov{\varphi}|_{\Ov{C}_{N}}:{\Ov{C}_{N}} \to B$ is an \'etale morphism,
$C_{N}^{\dagger}$ and $B^{\dagger}$ has no intersection points.
Hence we have 
\[
0=(2C_{0}^{\dagger}+b\Gamma^{\dagger})(C_{0}^{\dagger}+c\Gamma^{\dagger})=2(c-e)+b.
\]
Since $0 \leq c$ and $e<b$, we have $e<b\leq 2e$.
In particular $e$ is a positive integer.
Combine that $B^{\dagger}$ is an irreducible curve, we see that $b=2e$ and $c=0$.
Hence we have 
$B^{\dagger} \equiv 2C_{0}^{\dagger} +2e\Gamma^{\dagger}$ and $C_{N}^{\dagger} \equiv C_{0}^{\dagger}$.

To derive a contradiction, we assume $\Ov{R}$ has no singular points and $\Ov{\varphi}|_{\Ov{R}_{h}}$
is an unramified morphism.
Let $R_{h}^{\dagger}:=\pi_{\ast}\Ov{R}_{h}$ and put $R_{h}^{\dagger} \equiv \frac{r}{2}C_{0}^{\dagger} +s\Gamma^{\dagger}$ with an integer $s$.
The assumption that $\Ov{\varphi}|_{\Ov{R}_{h}}:{\Ov{R}_{h}} \to B$ is an \'etale morphism
implies $\Ov{R}_{h} \Ov{C}_{N} =0$ if $\Ov{C}_{N} \not\subset \Ov{R}_{h}$ or 
$(\Ov{R}_{h}-\Ov{C}_{N})\Ov{C}_{N} =0$ if $\Ov{C}_{N} \subset \Ov{R}_{h}$.
If $\Ov{C}_{N} \not\subset \Ov{R}_{h}$, we have $s=re/2$.
If $\Ov{C}_{N} \subset \Ov{R}_{h}$, we have $s=(r-2)e/2$.
Hence we see that $R_{h}^{\dagger} B^{\dagger}>0$ if $B^{\dagger} \not \subset R_{h}^{\dagger}$
and $(R_{h}^{\dagger} -B^{\dagger} )B^{\dagger}>0$ if $B^{\dagger} \subset R_{h}^{\dagger}$.
 by simple calculations.
It contradicts to that $\Ov{\varphi}|_{\Ov{R}_{h}}$ is an \'etale morphism.
\end{proof}

We close the section by the following remark.

\begin{lem}
\label{loc.triv.lem}
Assume that $n\geq3$ and let $f:S\to B$ be a primitive cyclic covering fibration of type $(g, 0, n)$.
If $R_{h}$ is \'etale over $B$, then $f$ is locally trivial and $R$ is $\PHI$-horizontal.
\end{lem}

\begin{proof}
Since $f$ is a primitive cyclic covering fibration of type $(g, 0, n)$,
the multiplicity of any singular point of $R$ is either in $n\Z$ or in $n\Z+1$.
From $n\geq3$, we conclude that $R$ has no double points. 
On the other hand, since $R_{h}$ is \'etale over $B$, $R_{h}$ is smooth and no fibers of $\varphi$ have a contact to $R_h$.
So we get $R=R_h$. Since $R$ is \'etale over $B$, we have $\chi_{f}=0$ (\cite{Eno}).
Thus, $f$ is locally trivial.
\end{proof}

\vspace{2\baselineskip}
\section{Estimations of the order with localized $K_{f}^2$}

Let $f:S\to B$ a primitive cyclic covering fibration of type $(g,0,n)$ with $r\geq 3n$ 
and assume that $f$ is not locally trivial.
We put
\begin{align}\label{eq(4.1)}
&\mu_{r,n}:=\frac{2n^{2}(r-1)}{(n-1)((n-1)r-2n)}.
\end{align}
Since $r=2(g+n-1)/(n-1)$, we can rewrite it as 
\begin{align*}
\mu_{g,n}=\left(1+\frac{1}{n-1}\right)^{2}\frac{2g+n-1}{g-1}
\end{align*}
as a function in $g,n$.

Let $K_{f}^{2}(\Ga_{p})$ be as in (\ref{localinv}).
We show that there exists a fiber $\Gamma_{p}$ such that 
\[
\frac{n\Sh \Ov{K}}{K_{f}^{2}(\Ga_{p})}\leq \mu_{r,n}
\]
or equivalently
\begin{align}
\label{eq(4.2)}
2n(r-1)K_{f}^{2}(\Ga_{p})&\geq (n-1)\left((n-1)r-2n) \right) \Sh \Ov{K}.
\end{align}
Let $\Ov{\varphi}:\Ov{P} \to B$ be the model as in Lemma~~\ref{hyoujyunka}
and let $\Ov{\Gamma}_{p}$ be a fiber satisfying $K_{f}^2(\Gamma_{p})>0$. 
We consider three cases separately:
\begin{itemize}
\item[4.1.] The first case. Assume that neither a singular point of $\Ov{R}$ or 
a ramification point of $\Ov{\varphi}|_{\Ov{R}_{h}}$ on $\Ov{\Gamma}_{p}$.
Note that $p \in \Delta$ in this assumption.
\item[4.2.] The second case. Assume that $\Ov{R}$ is smooth locally around $\Ov{\Gamma}_{p}$
and $\Ov{\varphi}|_{\Ov{R}_{h}}$ has a ramification point on $\Ov{\Gamma}_{p}$.
\item[4.3.] The third case. Assume that $\Ov{R}$ has a singular point on $\Ov{\Gamma}_{p}$.
\end{itemize}
Since the problem is local, we may consider $f:S\to B$ locally around $p$. 

\subsection{The first case}

\begin{prop}
\label{et.case}
Let the notation and the assumptions be as above.
Assume that neither a singular point of $\Ov{R}$ or 
a ramification point of $\Ov{\varphi}|_{\Ov{R}_{h}}$ on $\Ov{\Gamma}_{p}$.
Then we have 
\begin{align*}
 4nr(r-1)K_{f}^{2}(\Ga_{p}) =  \left((n^2-1)r^{2}-4n^{2}(r-1)\right)\sharp\Ov{K}.
\end{align*}
Furthermore,
\begin{align*}
 2n(r-1)K_{f}^{2}(\Ga_{p}) \geq (n-1)\left((n-1)r-2n\right) \sharp \Ov{K}
\end{align*}
if $\Ov{K} \cong \Z_{2}$.
\end{prop}
\begin{proof}
By $p \in \Delta$ and Lemma~\ref{badhyoujyunnka}, 
there exists one singular point of $R$ on $\Gamma_{p}$ of multiplicity $r/2$.
The assumption also implies that $\alpha_{0}(\Gamma_{p})=j_{1}(\Gamma_{p})=0$.   
By (\ref{localinv}), we have 
\begin{align*}
(r-1)K_{f}^{2}(\Ga_{p}) =  (n^2-1)\frac{r^{2}}{4n}-n(r-1).
\end{align*}
Since the assumption of $\Ov{R}$,
we have $r \geq \sharp \Ov{K}$.
If $ \Ov{K} \cong \Z_{2}$, we have $r \geq  2\sharp \Ov{K}$.
Noting that $\left((n^2-1)r^{2}-4n^{2}(r-1)\right) \geq (n-1)r \left((n-1)r-2n\right)$,
we have desired inequalities.
\end{proof}

\subsection{The second case}

\begin{prop}
\label{sm.case}
Let the notation and the assumptions be as above. 
If $\Ov{R}$ is smooth in a neighborhood of $\Ov{\Gamma}_{p}$ and is tangent to $\Ov{\Gamma}_p$, then 
it holds that 
\begin{align*}
2n(r-1)K_{f}^{2}(\Ga_{p})&\geq  (n-1)\left((n-1)r-2n) \right) \Sh \Ov{K}.
\end{align*}
In particular, $(\ref{eq(4.2)})$ holds true.
\end{prop}

\begin{proof}
Note that we have $\alpha_k(\Ga_p)=0$ for $k\geq 1$ and $j_{1}(\Ga_{p})=0$.
Furthermore, $\alpha_0(\Ga_p)$ is nothing but the degree of the ramification divisor of $R\to B$ over $p$.
Let $\Ov{z}\in \Ov{R}$ be a point of contact of $\Ov{R}$ and $\Ov{\Gamma}_{p}$
and let $\Stab_{\Ov{K}}(\Ov{z}) := \{\kappa \in \Ov{K} \mid \kappa(\Ov{z})=\Ov{z} \}$ be
the stabilizer of $\Ov{K}$ at $\Ov{z}$.

\smallskip

Case~1.\;The fiber $\Ov{\Gamma}_{p}$ has no node at $\Ov{z}$.

(1-1)\;$\Stab_{\KB}(\Ov{z})$ is not the unit group.

We know that $\Stab_{\KB}(\Ov{z})$ is a cyclic group, since it has $\Ov{z}$ as a fixed point.
Note that the fixed curve (that is, the curve traced by the fixed points on fibers) of $\Stab_{\KB}(\Ov{z})$ is a section of $\Ov{\varphi}:\Ov{P}\to B$ locally around $\Ov{z}$ and is not a local analytic branch of $\Ov{R}$, 
since $\Ov{R}$ is smooth and tangent to $\Ov{\Gamma}_{p}$ at $\Ov{z}$.
Then the $\Stab_{\KB}(\Ov{z})$-orbit of any point on $\Ov{R}$ sufficiently close to 
$\Ov{z}$ consists of $\sharp \Stab_{\KB}(\Ov{z})$ distinct points. 
Since $\Ov{R}$ and $\Ov{\Gamma}_{p}$ are preserved by the action of $\Stab_{\KB}(\Ov{z})$, 
we get $(\Ov{R}_{h} , \Ov{\Gamma}_{p})_{\Ov{z}}\geq \sharp \Stab_{\KB}(\Ov{z})$, 
where $(\Ov{R}_{h} , \Ov{\Gamma}_{p})_{\Ov{z}}$ denotes the local intersection number of $\Ov{R}$ 
and $\Ov{\Gamma}_p$ at $z$.
Note that $(\Ov{R}_{h} , \Ov{\Gamma}_{p})_{\Ov{z}}$ is the ramification index of $\Ov{R}_{h}\to B$ at $z$ in the present case. 
It follows that $\Ov{z}$ contributes at least  $(\sharp \Stab_{\KB}(\Ov{z})-1)$ to $\alpha_0(\Ga_p)$ and so does any point in the $\KB$-orbit $\KB\cdot \Ov{z}$ of $\Ov{z}$.
Hence
\begin{align*}
\alpha_{0}(\Ga_{p})\geq \Sh (\KB\cdot z)( \sharp \Stab_{\KB}(\Ov{z})-1)=\frac{\Sh \KB}{\sharp \Stab_{\KB}(\Ov{z})}(\sharp \Stab_{\KB}(\Ov{z})-1)
\end{align*}
and we infer readily from (\ref{localinv}) that 
\begin{align*}
2nr(r-1)K_{f}^{2}(\Ga_{p})\geq 2(n-1)r\bigr( (n-1)r-2n \bigr)\frac{(\sharp \Stab_{\KB}(\Ov{z})-1)}{\sharp \Stab_{\KB}(\Ov{z})}\Sh \KB.
\end{align*}
Since $\sharp \Stab_{\KB}(\Ov{z})\geq 2$, we get the desired inequality.

\smallskip

(1-2)\; $\Stab_{\KB}(\Ov{z})$ is the unit group.

By the assumption, we have $\Sh(\KB\cdot{\Ov{z}})=\sharp \KB$,
which implies that there exist at least $\sharp \KB$ points of contact in $\Ov{R}\cap \Ov{\Gamma}_p$.
Since the ramification index of $\Ov{\varphi}|_{\Ov{R}_{h}}:\Ov{R}_{h}\to B$ at each point of contact is not less than $2$, 
we have $\alpha_0(\Ov{\Ga}_p)\geq (2-1)\Sh\Ov{K}=\Sh\Ov{K}$ and 
\begin{align*}
2nr(r-1)K_{f}^{2}(\Ga_{p})\geq 2(n-1)r( (n-1)r-2n)\Sh\Ov{K}
\end{align*}
from (\ref{localinv}), which is stronger than what we want.

\smallskip

Case~2.\;The fiber $\Ov{\Gamma}_{p}$ has  node at $\Ov{z}$.

If $\Ov{K}\not\cong D_{2l}\;(l \geq 3)$, there exists a singular point of $\Ov{R}$ on $\Ov{z}$.
Hence we may assume that $\Ov{K}\cong\Z_{2}$ or $D_{4}$ in this case.
So we can obtain the desired inequality as in the proof of Proposition~\ref{et.case}.

\end{proof}

\subsection{The third case}

\noindent

Let $\Ov{\varphi}:\Ov{P} \to B$ be the model as in Lemma~~\ref{hyoujyunka}.
Take a singular point $\Ov{z} \in \Ov{R}$ and 
let $\Ov{\Gamma}_{p}$ be a fiber of $\Ov{\varphi}$ passing through $\Ov{z}$.
We introduce some notations as follows for $\Ov{z}$.

\begin{itemize}
\item[(1)] We denote by $\Rm{L}(\Ov{R},\Ov{z})$ and $\Rm{L}(\Ov{R}_{h},\Ov{z})$ a set of local analytic branches 
of $\Ov{R}$ and $\Ov{R}_{h}$ at $\Ov{z}$, respectively.
\item[(2)] Let $\Rm{L}_{\Rm{tr}}(\Ov{R},\Ov{z})$ and $\Rm{L}_{\Rm{tr}}(\Ov{R}_{h},\Ov{z})$ be a subset of $\Rm{L}(\Ov{R},\Ov{z})$ and $\Rm{L}(\Ov{R}_{h},\Ov{z})$ which consists of local analytic branches that meet $\Ov{\Gamma}_{p}$ transversally at $\Ov{z}$, respectively.
\item[(3)] Let $\Rm{L}_{\Rm{ta}}(\Ov{R},\Ov{z})$ and $\Rm{L}_{\Rm{ta}}(\Ov{R}_{h},\Ov{z})$ be a subset of $\Rm{L}(\Ov{R},\Ov{z})$ and $\Rm{L}(\Ov{R}_{h},\Ov{z})$ which consists of local analytic branches that tangent $\Ov{\Gamma}_{p}$ at $\Ov{z}$, respectively.
\item[(4)] For a irreducible component $\Ov{\Gamma}_{i}$ of $\Ov{\Gamma}_{p}$ passing through $\Ov{z}$,
let $\Rm{L}_{\Rm{ta}}(\Ov{R},\Ov{\Gamma}_{i},\Ov{z})$ and $\Rm{L}_{\Rm{ta}}(\Ov{R}_{h},\Ov{\Gamma}_{i},\Ov{z})$ be a subset of $\Rm{L}(\Ov{R},\Ov{z})$ and $\Rm{L}(\Ov{R}_{h},\Ov{z})$ which consists of local analytic branches that tangent $\Ov{\Gamma}_{i}$ at $\Ov{z}$, respectively.
If $\Ov{\Gamma}_{i} \subset \Ov{R}$, we define $\Ov{\Gamma}_{i} \in \Rm{L}_{\Rm{ta}}(\Ov{R},\Ov{\Gamma}_{i},\Ov{z})$.
\end{itemize}

Then we have disjoint decompositions
$\Rm{L}(\Ov{R},\Ov{z})=\Rm{L}_{\Rm{tr}}(\Ov{R},\Ov{z}) \sqcup \Rm{L}_{\Rm{ta}}(\Ov{R},\Ov{z})$ and
$ \Rm{L}(\Ov{R}_{h},\Ov{z})=\Rm{L}_{\Rm{tr}}(\Ov{R}_{h},\Ov{z}) \sqcup \Rm{L}_{\Rm{ta}}(\Ov{R}_{h},\Ov{z})$
by the definition.
To show (\ref{eq(4.2)}),
we consider four conditions as follows for $\Ov{\Ga}_{p}$ on which $\Ov{R}$ has a singular point.
\begin{itemize}
\item[(C1)] There exists a singular point $\Ov{z} \in \Ov{R}$ on $\Ov{\Gamma}_{p}$ such that $\Rm{Stab}_{\Ov{K}}(\Ov{z})=\{\Rm{id}\}$.

\item[(C2)] The condition $(C1)$ is not hold and there exist a singular point $\Ov{z}$ of $\Ov{R}$ on $\Ov{\Gamma}_{p}$ and an irreducible component $\Ov{\Gamma}_{i}$ of $\Ov{\Gamma}_{p}$ passing through 
$\Ov{z}$ such that $\Rm{L}_{\Rm{ta}}(\Ov{R}_{h},\Ov{\Gamma}_{i},\Ov{z}) \neq \emptyset$.

\item[(C3)] Conditions $(C1)$ and $(C2)$ are not hold and there exist a singular point $\Ov{z} \in \Ov{R}$ on $\Ov{\Gamma}_{p} $
and a local analytic branch $\Ov{D} \in \Rm{L}_{\Rm{tr}}(\Ov{R}_{h},\Ov{z})$ such that $\Ov{D}$ has 
cusp at $\Ov{z}$.

\item[(C4)] Conditions $(C1)$, $(C2)$ and $(C3)$ are not hold.
It implies that $\Rm{L}_{\Rm{ta}}(\Ov{R}_{h},\Ov{z})=\emptyset$
and any local analytic branch in $\Rm{L}_{\Rm{tr}}(\Ov{R}_{h},\Ov{z})$ is smooth at $\Ov{z}$ for any 
singular point $\Ov{z} \in \Ov{R}$.
\end{itemize}

\begin{prop}
\label{C1}
Assume $(C1)$ is hold.
Then it holds that 
\begin{align*}
 (r-1)K_{f}^{2}(\Ga_{p}) \geq 
  \left(-\frac{(n-1)}{n}\left((n-1)r-2n \right)B_{n}+
 (n^2-1)(r-n)-(r-1)n\right)\Sh\KB 
\end{align*}
\end{prop}

\begin{proof}
By the assumption, there exists a singular point $\Ov{z} \in \Ov{R}$ on $\Ov{\Gamma}_{p}$
such that $\sharp(\Ov{K}\cdot z)=\sharp\Ov{K}$.
It implies that there exist at least $\Sh\KB$ singular points of $\Ov{R}$
which are analytically equivalent to $\Ov{z}$.
Hence we get
\begin{align*}
 &(r-1)K_{f}^{2}(\Ga_{p})\geq  \left(-\frac{(n-1)}{n}\left((n-1)r-2n \right)B_{n}+
 (n^2-1)(r-n)-(r-1)n\right)\sharp \Ov{K} 
\end{align*}
from Lemma~\ref{lowerbdofK}.
\end{proof}

\smallskip

\begin{prop}
\label{C2}
Assume $(C2)$ is hold and $r \geq 18$ if $n=3$.
Then it holds that 
\begin{align*}
 2nr(r-1)K_{f}^{2}(\Ga_{p}) \geq  (n-1)r\left((n-1)r-2n\right)\sharp \Ov{K}.
\end{align*}
\end{prop}

\begin{proof}
Take $\Ov{z} \in \Ov{R}$ and $\Ov{\Gamma}_{1} \subset \Ov{\Gamma}_{p}$ such that  
$\Rm{L}_{\Rm{ta}}(\Ov{R}_{h},\Ov{\Gamma}_{1},\Ov{z})\neq \emptyset$.
Let $\Ov{D}_{1} \in \Rm{L}_{\Rm{ta}}(\Ov{R}_{h},\Ov{\Gamma}_{1},\Ov{z})$ be a local analytic branch 
that the proper transform of it and $\Ov{\Gamma}_{1}$ become to have no common tangent  
after the smallest number of blowing-ups over $\Ov{z}$ among $\Rm{L}_{\Rm{ta}}(\Ov{R}_{h},\Ov{\Gamma}_{1},\Ov{z})$.
Let $\Rm{Stab}_{\Ov{K}}(\Ov{\Gamma}_{1},\Ov{z})
:=\{\kappa \in \Rm{Stab}_{\Ov{K}}(\Ov{z})\mid \kappa(\Ov{\Gamma}_{1})=\Ov{\Gamma}_{1}\}$ and 
$\Ov{\Bb{D}}_{1}$ be the $\Rm{Stab}_{\Ov{K}}(\Ov{\Gamma}_{1},\Ov{z})$-orbit of $\Ov{D}_{1}$.
We introduce some notations as follows.
\begin{itemize}
\item $m_{\Ov{D}_{1}}:=\Rm{mult}_{\Ov{z}}\Ov{D}_{1}$
\item $m_{\Ov{\Bb{D}}_{1}}:=\Rm{mult}_{\Ov{z}}\Ov{\Bb{D}}_{1}$
\item $m:=\Rm{min}\{m' \mid m' \geq m_{\Ov{\Bb{D}}},\quad m' \in n\Z$ or $n\Z + 1\}$
\end{itemize}
Assume that we need $v$ times of blowing-ups until the proper transform of $\Ov{D}_{1}$ is not tangent to that of $\Ov{\Gamma}_{1}$.
Then we have
\begin{align*} 
(\Ov{D}_{1},\Ov{\Gamma}_{1})_{\Ov{z}} &=vm_{\Ov{D}_{1}}
+u_{\Ov{D}_{1}} \quad (1 \leq u_{\Ov{D}_{1}} \leq m_{\Ov{D}_{1}}),\\
(\Ov{\Bb{D}}_{1},\Ov{\Gamma}_{1})_{\Ov{z}} &=vm_{\Ov{\Bb{D}}_{1}}
+u_{\Ov{\Bb{D}}_{1}} \quad ( u_{\Ov{\Bb{D}}_{1}} \leq m_{\Ov{\Bb{D}}_{1}}),
\end{align*}
where $u_{\Ov{D}_{1}}$ and $u_{\Ov{\Bb{D}}_{1}}$ denotes the multiplicity 
of the proper transform of $\Ov{D}_{1}$ and $\Ov{\Bb{D}}_{1}$ at the point of contact to the exceptional curve of the $v$-th blowing-up, respectively.
Put $u:=\Rm{min}\{u' \mid u' \geq u_{\Ov{\Bb{D}}},\quad u' \in n\Z$ or $n\Z + 1\}$.

\smallskip

Case1.\;The fiber $\Ov{\Gamma}_{p}$ has no node at $\Ov{z}$.

We note that $\Rm{Stab}_{\Ov{K}}(\Ov{z})=\Rm{Stab}_{\Ov{K}}(\Ov{\Gamma}_{1},\Ov{z})$ in this case.
Put $\Ov{K} \cdot \Ov{z} := \{ \Ov{z}_{1}(:=\Ov{z}), \cdots , \Ov{z}_{\Sh(\Ov{K}\cdot \Ov{z})}\}$.
Let $\{\Ov{\Bb{D}}_{1},\cdots, \Ov{\Bb{D}}_{\sharp(\Ov{K}\cdot \Ov{z})}\}$ be the $\Ov{K}$-orbit of $\Ov{\Bb{D}}_{1}$ such that $\Ov{z}_{j} \in \Ov{\Bb{D}}_{j}$.
Since $\Ov{z}_{j} \in \Ov{\Bb{D}}_{j}$ is analytically equivalent to $\Ov{z}_{1} \in \Ov{\Bb{D}}_{1}$,
the proper transform of $\Ov{\Bb{D}}_{j}$ is not tangent to 
the proper transform of $\Ov{\Gamma}_{p}$ after $v$-times blowing-ups over $\Ov{z}_{j}$.
We denote by $\Ha{z}_{j}$ the intersection point of the proper transform $\Ov{\Bb{D}}_{j}$ and that of $\Ov{\Gamma}_{p}$ after $v$-times blowing-ups over $\Ov{z}_{j}$, respectively.
Let $\Ha{P}\to \Ov{P}$ be the composite of the above $v$-times blowing-ups at $\Ov{z}_{j}$ for $j=1,\cdots, \sharp (\Ov{K}\cdot z)$.
We denote by $\Ch{\psi}:\Ha{P}\to P$ the composite of the above $\Ha{P}\to \Ov{P}$ and $\Ov{P}\to P$.
Let $\Ha{R}$ be the branch locus on $\Ha{P}$ and
 $\Ha{\Gamma}_{p}$ be a fiber of $\Ha{\varphi}:=\varphi \circ \Ch{\psi}$ over $p$.
 
\smallskip

(1-1)\;The branch locus $\Ha{R}$ has a singular point $\Ha{z}_{1}$.

We note that $\Ha{R}$ has a singular point $\Ha{z}_{1}, \cdots ,\Ha{z}_{\sharp(\Ov{K}\cdot \Ov{z})}$ in this case.
There exist $\Sh(\KB\cdot \Ov{z})v$ singular points of $\Ov{R}$ 
(including infinitely near one) of multiplicity at least $m$
and $\sharp(\Ov{K}\cdot \Ov{z})$ singular points of multiplicity at least $u$.
Hence we have
\begin{align*}
 (r-1)K_{f}^{2}(\Ga_{p})
 &\geq  \frac{n-1}{n}\left((n-1)r-2n \right)\sum_{\Ha{z}\in \Ha{\Gamma}_{p}}\Al^{+}_{0}(\Gamma_{p})_{\Ha{z}} \\
& + \biggl( -2\frac{n-1}{n}\bigr((n-1)r-2n \bigl)
+(n^2 -1)\left[\frac{u}{n}\right]\left(r-n\left[\frac{u}{n}\right]\right)-n(r-1) \biggr)\sharp(\Ov{K}\cdot \Ov{z})\\
 &-2\frac{n-1}{n}\left((n-1)r-2n\right)\sharp\Ha{R}_{v}(p)\\
 &+\left((n^2 -1)\left[\frac{m}{n}\right]\left(r-n\left[\frac{m}{n}\right]\right) -n(r-1) \right)\sharp(\Ov{K}\cdot \Ov{z})v\\
 \end{align*}
from Lemma~\ref{addlem}.
Since $(\Ov{D}_{1},\Ov{\Gamma}_{p})_{\Ov{z}_{1}} \geq 2$,
we can check $\Al_{0}^{+}(\Gamma_{p})_{\Ha{z}_{1}}\geq (\Ov{\Bb{D}}_{1},\Ov{\Gamma}_{p})_{\Ov{z}_{1}} / 2$.
Hence we have 
\begin{align*}
 &(r-1)K_{f}^{2}(\Ga_{p}) \\
 &\geq \frac{n-1}{n}\left((n-1)r-2n \right)
\frac{(\Ov{\Bb{D}}_{1},\Ov{\Gamma}_{p})_{\Ov{z}_{1}}}{2} \sharp(\Ov{K}\cdot \Ov{z})\\
& + \biggl( -2\frac{n-1}{n}\bigr((n-1)r-2n \bigl)
+(n^2 -1)\left[\frac{u}{n}\right]\left(r-n\left[\frac{u}{n}\right]\right)-n(r-1) \biggr)\sharp(\Ov{K}\cdot \Ov{z}) \\
 &-2\frac{n-1}{n}\left((n-1)r-2n\right)\sharp\Ha{R}_{v}(p)
 +\left((n^2 -1)\left[\frac{m}{n}\right]\left(r-n\left[\frac{m}{n}\right]\right) -n(r-1) \right)\sharp(\Ov{K}\cdot \Ov{z})v.
 \end{align*}
We need to show 
\begin{align*}
2nr&\biggl\{ \frac{n-1}{n}\left((n-1)r-2n \right)
\frac{(\Ov{\Bb{D}}_{1},\Ov{\Gamma}_{p})_{\Ov{z}_{1}}}{2} \sharp(\Ov{K}\cdot \Ov{z})\biggr.\\
& + \biggl( -2\frac{n-1}{n}\bigr((n-1)r-2n \bigl)
+(n^2 -1)\left[\frac{u}{n}\right]\left(r-n\left[\frac{u}{n}\right]\right)-n(r-1) \biggr)\sharp(\Ov{K}\cdot \Ov{z}) \\
 &-2\frac{n-1}{n}\left((n-1)r-2n\right)\sharp\Ha{R}_{v}(p)
 \biggr.+\left((n^2 -1)\left[\frac{m}{n}\right]\left(r-n\left[\frac{m}{n}\right]\right) -n(r-1) \right)\sharp(\Ov{K}\cdot \Ov{z})v\biggl\}     \\
& -(n-1)r\left((n-1)r-2n\right)\sharp \Ov{K} \geq 0
 \end{align*}
in order to obtain the desired inequality. 
By $(\Ov{\Bb{D}}_{1},\Ov{\Gamma}_{p})_{\Ov{z}_{1}} \geq \sharp \Rm{Stab}_{\Ov{K}}(\Ov{z}_{1})$,
it is sufficient to show 
\begin{align*}
2nr&\biggl\{ \frac{n-1}{n}\left((n-1)r-2n \right)
\frac{(\Ov{\Bb{D}}_{1},\Ov{\Gamma}_{p})_{\Ov{z}_{1}}}{2} \biggr.\\
& -2\frac{n-1}{n}\left((n-1)r-2n\right)\left(\frac{\sharp\Ha{R}_{v}(p)}{\sharp(\Ov{K}\cdot \Ov{z})}+1\right) \\
& +\biggl( (n^2 -1)\left[\frac{u}{n}\right]\left(r-n\left[\frac{u}{n}\right]\right)-n(r-1) \biggr)
 \biggr.+\left((n^2 -1)\left[\frac{m}{n}\right]\left(r-n\left[\frac{m}{n}\right]\right) -n(r-1) \right)v\biggl\}     \\
& -(n-1)r\left((n-1)r-2n\right)(\Ov{\Bb{D}}_{1},\Ov{\Gamma}_{p})_{\Ov{z}_{1}} \geq 0.
 \end{align*}
It is equivalent to 
\begin{align*}
  2nr& \biggl\{ -2\frac{n-1}{n}\left((n-1)r-2n\right)\left(\frac{\sharp\Ha{R}_{v}(p)}{\sharp(\Ov{K}\cdot \Ov{z})}+1\right)\\
 &+ \biggl( (n^2 -1)\left[\frac{u}{n}\right]\left(r-n\left[\frac{u}{n}\right]\right)-n(r-1) \biggr) \\
& \biggr.+\left((n^2 -1)\left[\frac{m}{n}\right]\left(r-n\left[\frac{m}{n}\right]\right) -n(r-1) \right)v\biggl\}   \geq 0.
 \end{align*}
The number of irreducible components of $\Ha{R}_{v}(p)$ which appear in $\Ha{P} \to \Ov{P}$ is at least 
$2\sharp (\Ov{K}\cdot \Ov{z})$.
This can be seen as follows:
The number of blowing-ups to obtain each irreducible component of $\Ti{R}_{v}(p)$ is at least $n$
by Remark~\ref{basic}.
Thus, for each irreducible component of  $\Ha{R}_{v}(p)$
there exists a local analytic branch in $\Rm{L}(\Ov{R}_{h},\Ov{z})$ which is tangent to it.
By the definition of $\Ha{P} \to \Ov{P}$, any local analytic branch in $\Rm{L}(\Ov{R}_{h},\Ov{z})$ can be tangent to only exceptional curves of first and last blowing-ups at most.
Thus, it is at least $2\sharp (\Ov{K}\cdot \Ov{z})$.
It implies that 
\[
\sharp \Ha{R}_{v}(p) \leq
   \begin{cases}
\sharp \Ov{\Gamma}_{p}+ \sharp (\Ov{K}\cdot \Ov{z}) & (\Rm{if\;}v=1),\\
\sharp \Ov{\Gamma}_{p}+ 2\sharp (\Ov{K}\cdot \Ov{z}) & (\Rm{if\;}v\geq2). \\
 \end{cases}
\]
Since $\sharp (\Ov{K}\cdot \Ov{z})=2$ if $\sharp \Ov{\Gamma}_{p}=2$, we get 
\[
-\left( \frac{\sharp\Ha{R}_{v}(p)}{\sharp(\Ov{K}\cdot \Ov{z})}+1 \right) \geq
   \begin{cases}
    -3  & (\Rm{if\;}v = 1), \\
    -4  & (\Rm{if\;}v \geq 2).
     \end{cases}
\]
If $v\geq 2$, it is sufficient to show 
\begin{align*}
2nr&\biggl\{  \biggl( (n^2 -1)\left[\frac{u}{n}\right]\left(r-n\left[\frac{u}{n}\right]\right)-n(r-1) \biggr)\\
 &-8\frac{n-1}{n}\left((n-1)r-2n\right)
 \biggr.+\left((n^2 -1)\left[\frac{m}{n}\right]\left(r-n\left[\frac{m}{n}\right]\right) -n(r-1) \right)2\biggl\}  \geq 0
 \end{align*}
 to obtain the desired inequality.
 The left hand side of the above inequality can be considered as a quadratic function in $m$ and $u$ defined in the interval $[n , r/2]$.
Hence it suffices to consider its value at $m=u=n$.
 Therefore it is sufficient to show 
 \begin{align*}
 3(n^2-1)(r-n)-n(r-1) -8\frac{(n-1)}{n}\left((n-1)r-2n\right) \geq 0.
 \end{align*}
We can check it by a simple calculation if $n\geq4$ or $n=3$ and $r \geq 12$.
If $v=1$, we can show it similarly.

\smallskip

(1-2)\;The branch locus $\Ha{R}$ is smooth at $\Ha{z}_{1}$.

We note that $\Ha{R}$ is smooth at $\Ha{z}_{1}, \cdots, \Ha{z}_{\sharp(\Ov{K}\cdot \Ov{z})}$ 
and $\Ov{\Bb{D}}_{1}=\Ov{D}_{1}$ in this case.
There exist $\Sh(\KB\cdot \Ov{z})v$ singular points of $\Ov{R}$ 
(including infinitely near one) of multiplicity at least $m$.
Hence we have 
\begin{align*}
 (r-1)K_{f}^{2}(\Ga_{p})
 &\geq \frac{n-1}{n}\left((n-1)r-2n \right)
\frac{(\Ov{\Bb{D}}_{1},\Ov{\Gamma}_{p})_{\Ov{z}_{1}}}{2} \sharp(\Ov{K}\cdot \Ov{z})\\
 &-2\frac{n-1}{n}\left((n-1)r-2n\right)\sharp\Ha{R}_{v}(p) \\
 &+\left((n^2 -1)\left[\frac{m}{n}\right]\left(r-n\left[\frac{m}{n}\right]\right) -n(r-1) \right)\sharp(\Ov{K}\cdot \Ov{z})v
 \end{align*}
 from Lemma~\ref{addlem}.
We need to show 
\begin{align*}
2nr&\biggl\{ \frac{n-1}{n}\left((n-1)r-2n \right)
\frac{(\Ov{\Bb{D}}_{1},\Ov{\Gamma}_{p})_{\Ov{z}_{1}}}{2} \biggr.\\
 &-2\frac{n-1}{n}\left((n-1)r-2n\right)\frac{\sharp\Ha{R}_{v}(p)}{\sharp(\Ov{K}\cdot \Ov{z})}
 \biggr.+\left((n^2 -1)\left[\frac{m}{n}\right]\left(r-n\left[\frac{m}{n}\right]\right) -n(r-1) \right)v\biggl\} \\
& -(n-1)r\left((n-1)r-2n\right)(\Ov{\Bb{D}}_{1},\Ov{\Gamma}_{p})_{\Ov{z}_{1}} \geq 0
 \end{align*}
in order to obtain the desired inequality. 
It is equivalent to show 
\begin{align*}
2nr&\biggl\{  -2\frac{n-1}{n}\left((n-1)r-2n\right)\frac{\sharp\Ha{R}_{v}(p)}{\sharp(\Ov{K}\cdot \Ov{z})}
 \biggr.+\left((n^2 -1)\left[\frac{m}{n}\right]\left(r-n\left[\frac{m}{n}\right]\right) -n(r-1) \right)v\biggl\} \geq 0.
 \end{align*}
Since $\Ha{R}$ is smooth at $\Ha{z}_{j}$,
the exceptional curve of $v$-th blowing-up over $\Ov{z}_{j}$ is not contained in $\Ha{R}_{v}(p)$ 
for each $j$.
It implies that
\[
\sharp \Ha{R}_{v}(p) \leq
   \begin{cases}
\sharp \Ov{\Gamma}_{p}& (\Rm{if\;}v=1),\\
\sharp \Ov{\Gamma}_{p}+ \sharp (\Ov{K}\cdot \Ov{z}) & (\Rm{if\;}v\geq2), \\
 \end{cases}
\]
Since $\sharp (\Ov{K}\cdot \Ov{z})=2$ if $\sharp \Ov{\Gamma}_{p}=2$, we get 
\[
- \frac{\sharp\Ha{R}_{v}(p)}{\sharp(\Ov{K}\cdot \Ov{z})} \geq
   \begin{cases}
    -1  & (\Rm{if\;}v = 1), \\
    -2  & (\Rm{if\;}v \geq 2).
     \end{cases}
\]
Hence it suffices to show 
\begin{align*}
(n^2 -1)(r-n)-n(r-1)
-2\frac{n-1}{n}\left((n-1)r-2n\right) \geq 0.
\end{align*}
We can check it by a simple calculations.

\smallskip

Case2.\;$\Ov{\Gamma}_{p}$ has a node at $\Ov{z}$ which consists with two irreducible components 
$\Ov{\Gamma}_{1}$, $\Ov{\Gamma}_{2}$.

We recall that $\Ov{K}$ has an action $\kappa_{\Rm{sw}}$ 
which switches $\Ov{\Gamma}_{1}$ and $\Ov{\Gamma}_{2}$.
Put $\Ov{D}_{2}:=\kappa_{\Rm{sw}}(\Ov{D}_{1})$ and $\Ov{\Bb{D}}_{2}:=\kappa_{\Rm{sw}}(\Ov{\Bb{D}}_{1})$.
Since $\Ov{D}_{1}$ has a cusp at $\Ov{z}$, no action of $\Ov{K}$ stables $\Ov{D}_{1}$.
Thus, we have 
$(\Ov{K}\cdot \Ov{D},\Ov{\Gamma}_{p})_{\Ov{z}}=(\Ov{\Bb{D}}_{1},\Ov{\Gamma}_{p})_{\Ov{z}}+(\Ov{\Bb{D}}_{2},\Ov{\Gamma}_{p})_{\Ov{z}} \geq \sharp \Ov{K}$, 
where $\Ov{K}\cdot \Ov{D}$ denotes the $\Ov{K}$-orbit of $\Ov{D}$.
Hence we have $2(v+1)m_{\Ov{\Bb{D}}_{1}}+2u_{\Ov{\Bb{D}}_{1}}=2(\Ov{\Bb{D}}_{1},\Ov{\Gamma}_{p})_{\Ov{z}} \geq \sharp \Ov{K}$.
We denote by $\Ha{z}_{j}$ the intersection point of the 
proper transform of $\Ov{\Bb{D}}_{j}$ and that of $\Ov{\Gamma}_{j}$ 
after $v$-times blowing-ups over $\Ha{z}_{j}$ for each $j$.
Let $\Ha{P}\to \Ov{P}$ be the composite of the above ($2v-1$)-times blowing-ups. 

\smallskip

(2-1)\;The branch locus $\Ha{R}$ is singular at $\Ha{z}_{1}$ and $\Ha{z}_{2}$.

There exist one singular point of multiplicity $r/2$,
$(2v-1)$ singular  points of multiplicity at least $m$ 
and $2$ singular points of multiplicity at least $u$.
Hence we have 
\begin{align*}
 (r-1)K_{f}^{2}(\Ga_{p})  
&\geq \sum_{\Ha{z}\in\Ha{\Gamma}_{p}}\left(\frac{n-1}{n}\left((n-1)r-2n\right)\alpha^{+}_{0}(\Ga_{p})_{\Ha{z}}\right)\\
&+\frac{(n^2-1)}{4n}r^{2}-n(r-1) \\
&+\left( (n^2-1)\left[\frac{m}{n}\right](r-n\left[\frac{m}{n}\right])
-n(r-1)\right)(2v-1)\\
 &+2\left((n^2-1)\left[\frac{u}{n}\right](r-n\left[\frac{u}{n}\right])
-n(r-1)\right)\\
&-2\frac{n-1}{n}\left((n-1)r-2n\right)(\sharp \Ha{R}_{v}(p)+2).
\end{align*} 
from Lemma~~\ref{addlem}.
Since $(\Ov{D}_{1},\Ov{\Gamma}_{p})_{\Ov{z}} = (\Ov{D}_{2},\Ov{\Gamma}_{p})_{\Ov{z}} \geq 2$,
we have 
\begin{align*}
\sum_{\Ha{z}\in\Ha{\Gamma}_{p}} \Al_{0}^{+}(\Gamma_{p})_{\Ha{z}}
\geq \frac{(\Ov{\Bb{D}}_{1},\Ov{\Gamma}_{p})_{\Ov{z}}}{2} + \frac{ (\Ov{\Bb{D}}_{2},\Ov{\Gamma}_{p})_{\Ov{z}}}{2} = (\Ov{\Bb{D}}_{1},\Ov{\Gamma}_{p})_{\Ov{z}}.
\end{align*}
Hence we have 
\begin{align*}
 (r-1)K_{f}^{2}(\Ga_{p})  
&\geq \frac{n-1}{n}\left((n-1)r-2n\right)(\Ov{\Bb{D}}_{1},\Ov{\Gamma}_{p})_{\Ov{z}}
+\frac{(n^2-1)}{4n}r^{2}-n(r-1)\\
&+\left( (n^2-1)\left[\frac{m}{n}\right](r-n\left[\frac{m}{n}\right])
-n(r-1)\right)(2v-1)\\
 &+2\left((n^2-1)\left[\frac{u}{n}\right](r-n\left[\frac{u}{n}\right])
-n(r-1)\right)\\
&-2\frac{n-1}{n}\left((n-1)r-2n\right)(\sharp \Ha{R}_{v}(p)+2).
\end{align*} 
We need to show 
\begin{align*}
2nr& \biggl\{\frac{n-1}{n}\left((n-1)r-2n\right)(\Ov{\Bb{D}}_{1},\Ov{\Gamma}_{p})_{\Ov{z}} 
+\frac{(n^2-1)}{4n}r^{2}-n(r-1) \biggr.\\
&+\left( (n^2-1)\left[\frac{m}{n}\right](r-n\left[\frac{m}{n}\right])
-n(r-1)\right)(2v-1)\\
 &+2\left((n^2-1)\left[\frac{u}{n}\right](r-n\left[\frac{u}{n}\right])
-n(r-1)\right)\\
\biggl.&-2\frac{n-1}{n}\left((n-1)r-2n\right)(\sharp \Ha{R}_{v}(p)+2)\biggr\} \\
&-(n-1)r\left((n-1)r-2n\right)2(\Ov{\Bb{D}}_{1},\Ov{\Gamma}_{p})_{\Ov{z}} \geq 0
\end{align*} 
in order to obtain the desired inequalities.
It is equivalent to 
\begin{align*}
2nr& \biggl\{\frac{(n^2-1)}{4n}r^{2}-n(r-1) +\left( (n^2-1)\left[\frac{m}{n}\right](r-n\left[\frac{m}{n}\right])
-n(r-1)\right)(2v-1)\biggr.\\
 &+2\left((n^2-1)\left[\frac{u}{n}\right](r-n\left[\frac{u}{n}\right])
-n(r-1)\right)\\
\biggl.&-2\frac{n-1}{n}\left((n-1)r-2n\right)(\sharp \Ha{R}_{v}(p)+2)\biggr\}\geq 0.
\end{align*} 
We can check that 
\[
-\left( \sharp\Ha{R}_{v}(p)+2\right)\geq
   \begin{cases}
    -5  & (\Rm{if\;}v = 1), \\
    -7  & (\Rm{if\;}v \geq 2),
     \end{cases}
\]
by the similar argument in (1-1).
Hence it is sufficient to show 
\begin{align*}
3(n^2-1)(r-n)-n(r-1) -10\frac{(n-1)}{n}\left((n-1)r-2n\right) \geq 0.
\end{align*} 
We can show it by a simple calculation if $n\geq 4$ or $n=3$ and $r \geq 18$. 

\smallskip

(2-2)\;The branch locus $\Ha{R}$ is smooth at $\Ha{z}_{1}$ and $\Ha{z}_{2}$.

There exist one singular point of multiplicity $r/2$,
$(2v-1)$ singular  points of multiplicity at least $m$.
Hence we have 
\begin{align*}
 (r-1)K_{f}^{2}(\Ga_{p})  
&\geq \frac{n-1}{n}\left((n-1)r-2n\right)(\Ov{\Bb{D}}_{1},\Ov{\Gamma}_{p})_{\Ov{z}}
+\frac{(n^2-1)}{4n}r^{2}-n(r-1)\\
&+\left( (n^2-1)\left[\frac{m}{n}\right](r-n\left[\frac{m}{n}\right])
-n(r-1)\right)(2v-1)\\
&-2\frac{n-1}{n}\left((n-1)r-2n\right)\sharp \Ha{R}_{v}(p).
\end{align*} 
from Lemma~~\ref{addlem}.
We need to show 
\begin{align*}
2nr& \biggl\{\frac{n-1}{n}\left((n-1)r-2n\right)(\Ov{\Bb{D}}_{1},\Ov{\Gamma}_{p})_{\Ov{z}} 
+\frac{(n^2-1)}{4n}r^{2}-n(r-1) \biggr.\\
&+\left( (n^2-1)\left[\frac{m}{n}\right](r-n\left[\frac{m}{n}\right])
-n(r-1)\right)(2v-1)\\
\biggl.&-2\frac{n-1}{n}\left((n-1)r-2n\right)\sharp \Ha{R}_{v}(p)\biggr\} \\
&-(n-1)r\left((n-1)r-2n\right)2(\Ov{\Bb{D}}_{1},\Ov{\Gamma}_{p})_{\Ov{z}} \geq 0
\end{align*} 
in order to obtain the desired inequalities.
It is equivalent to 
\begin{align*}
2nr& \biggl\{\frac{(n^2-1)}{4n}r^{2}-n(r-1) +\left( (n^2-1)\left[\frac{m}{n}\right](r-n\left[\frac{m}{n}\right])
-n(r-1)\right)(2v-1)\biggr.\\
\biggl.&-2\frac{n-1}{n}\left((n-1)r-2n\right)\sharp \Ha{R}_{v}(p)\biggr\} \geq 0.
\end{align*} 
We can check that 
\[
- \sharp\Ha{R}_{v}(p)\geq
   \begin{cases}
    0  & (\Rm{if\;}v = 1), \\
    -1  & (\Rm{if\;}v \geq 2).
     \end{cases}
\]
by the similar argument in (1-1).
If $v \geq 2$, it is sufficient to show 
\begin{align*}
(n^2-1)(r-n)-n(r-1) -2\frac{(n-1)}{n}\left((n-1)r-2n\right) \geq 0.
\end{align*} 
We can show it by a simple calculation. 
\end{proof}

\begin{prop}
\label{C3}
Assume $(C3)$ is hold and $r \geq 12$ if $n=3$.
Then it holds that 
\begin{align*}
 2nr(r-1) K_{f}^{2}(\Ga_{p}) \geq 
  (n-1)r\left((n-1)r-2n\right) \sharp{\Ov{K}} .
\end{align*}
\end{prop}

\begin{proof}
Let $\Ov{z} \in \Ov{R}$ and $\Ov{D} \in \Rm{L}_{\Rm{tr}}(\Ov{R},\Ov{z})$ be as in (C3).
Denote by $\Ov{\Bb{D}}$ the $\Rm{Stab}_{\Ov{K}}(\Ov{z})$-orbit of $\Ov{D}$.
We introduce some notations as follows.
\begin{itemize}
\item $m_{\Ov{D}}:=\Rm{mult}_{\Ov{z}}\Ov{D}$
\item $m_{\Ov{\Bb{D}}}:=\Rm{mult}_{\Ov{z}}\Ov{\Bb{D}}$
\item $m:=\Rm{min}\{m' \mid m' \geq m_{\Ov{\Bb{D}}},\quad m' \in n\Z$ or $n\Z + 1\}$.
\end{itemize}

\smallskip 

Case1.\;The fiber $\Ov{\Gamma}_{p}$ has a node at $\Ov{z}$.

Put $\Ov{K} \cdot \Ov{z} := \{ \Ov{z}_{1}(:=\Ov{z}), \cdots , \Ov{z}_{\Sh(\Ov{K}\cdot \Ov{z})}\}$.
Let $\Ha{P}\to \Ov{P}$ be the composite of blowing-ups at $\Ov{z}_{j}$ for $j=1,\cdots, \sharp (\Ov{K}\cdot z)$.
We denote by $\Ch{\psi}:\Ha{P}\to P$ the composite of the above $\Ha{P}\to \Ov{P}$ and $\Ov{P}\to P$.
Let $\Ha{R}$ be the branch locus on $\Ha{P}$ and
 $\Ha{\Gamma}_{p}$ be a fiber of $\Ha{\varphi}:=\varphi \circ \Ch{\psi}$ over $p$.
By $(\Ov{\Bb{D}},\Ov{\Gamma}_{p})_{\Ov{z}} \geq 2$, we have
\[
\sum_{\Ha{z}\in \Ha{\Gamma}_{p}}\Al^{+}_{0}(\Gamma_{p})_{\Ha{z}}
\geq \frac{1}{2}(\Ov{\Bb{D}},\Ov{\Gamma}_{p})_{\Ov{z}}  \sharp (\Ov{K}\cdot \Ov{z}).
\]
Since there exist $\sharp (\Ov{K}\cdot \Ov{z})$ singular points of $\Ov{R}$ of multiplicity at least $m$, 
we have 
\begin{align*}
 (r-1)K_{f}^{2}(\Ga_{p})_{\Ha{z}} 
 &\geq \frac{n-1}{n}\left((n-1)r-2n \right)
 \frac{1}{2}(\Ov{\Bb{D}},\Ov{\Gamma}_{p})_{\Ov{z}} \sharp(\Ov{K}\cdot \Ov{z})\\
&+\left((n^2 -1)\left[\frac{m}{n}\right]\left(r-n\left[\frac{m}{n}\right]\right) -n(r-1) \right)\sharp(\Ov{K}\cdot \Ov{z}).\\
 &-2\frac{n-1}{n}\left((n-1)r-2n\right)\sharp\Ha{R}_{v}(p)
 \end{align*}
from Lemma~\ref{addlem}.
We need to show 
\begin{align*}
2nr&\biggl\{ \frac{n-1}{n}\left((n-1)r-2n \right)
 \frac{1}{2}(\Ov{\Bb{D}},\Ov{\Gamma}_{p})_{\Ov{z}}  \sharp(\Ov{K}\cdot \Ov{z})\biggr.\\
&+\left((n^2 -1)\left[\frac{m}{n}\right]\left(r-n\left[\frac{m}{n}\right]\right) -n(r-1) \right)\sharp(\Ov{K}\cdot \Ov{z}) \\
\biggl.  &-2\frac{n-1}{n}\left((n-1)r-2n\right)\sharp\Ha{R}_{v}(p)\biggr\}\\
& -(n-1)r\left((n-1)r-2n\right)\sharp \Ov{K} \geq 0
 \end{align*}
in order to obtain the desired inequality. 
By $(\Ov{\Bb{D}},\Ov{\Gamma}_{p})_{\Ov{z}} \geq  \sharp \Rm{Stab}_{\Ov{K}}(\Ov{z})$, it is sufficient to show
\begin{align*}
2nr\biggl\{ (n^2 -1)\left[\frac{m}{n}\right]\left(r-n\left[\frac{m}{n}\right]\right) -n(r-1) 
-2\frac{n-1}{n}\left((n-1)r-2n\right)\frac{\sharp\Ha{R}_{v}(p)}{\sharp(\Ov{K}\cdot \Ov{z})}\biggr\}\geq 0.
 \end{align*} 
By $n \geq 3$,
it cannot be $\Rm{mult}_{\Ov{z}}\Ov{R}\in n\Z+1$ and $\Ov{\Gamma}_{p} \subset \Ov{R}$ at the same time.
Hence we have $\sharp\Ha{R}_{v}(p)/\sharp(\Ov{K}\cdot \Ov{z}) \leq 1$.
It suffices to show 
\begin{align*}
(n^2 -1)(r-n) -n(r-1)  
 -2\frac{n-1}{n}\left((n-1)r-2n\right) \geq 0.
 \end{align*}
We can check it by a simple calculation.

\smallskip

Case2.\;$\Ov{\Gamma}_{p}$ has a node at $\Ov{z}$ which consists with two irreducible components 
$\Ov{\Gamma}_{1}$, $\Ov{\Gamma}_{2}$.

Let $\Ha{P}\to \Ov{P}$ be the blowing-up at $\Ov{z}$.
We denote by $\Ch{\psi}:\Ha{P}\to P$ the composite of the above $\Ha{P}\to \Ov{P}$ and $\Ov{P}\to P$.
Let $\Ha{R}$ be the branch locus on $\Ha{P}$ and
 $\Ha{\Gamma}_{p}$ be a fiber of $\Ha{\varphi}:=\varphi \circ \Ch{\psi}$ over $p$.
We have $(\Ov{\Bb{D}},\Ov{\Gamma}_{p})_{\Ov{z}}\geq \sharp \Ov{K}$.
There exist one singular point of multiplicity $r/2$,
the singular point $\Ov{z}$ of multiplicity at least $m$.
Hence we have
\begin{align*}
 (r-1)K_{f}^{2}(\Ga_{p})  
&\geq \frac{n-1}{n}\left((n-1)r-2n\right)\frac{(\Ov{\Bb{D}}_{1},\Ov{\Gamma}_{p})_{\Ov{z}}}{2}
+\frac{(n^2-1)}{4n}r^{2}-n(r-1)\\
&+\left( (n^2-1)\left[\frac{m}{n}\right](r-n\left[\frac{m}{n}\right])
-n(r-1)\right)\\
 &-2\frac{n-1}{n}\left((n-1)r-2n\right)\sharp \Ha{R}_{v}(p).\\
\end{align*} 
We need to show 
\begin{align*}
2nr& \biggl\{\frac{n-1}{n}\left((n-1)r-2n\right)\frac{(\Ov{\Bb{D}}_{1},\Ov{\Gamma}_{p})_{\Ov{z}} }{2}
+\frac{(n^2-1)}{4n}r^{2}-n(r-1) \biggr.\\
&+\left( (n^2-1)\left[\frac{m}{n}\right](r-n\left[\frac{m}{n}\right])
-n(r-1)\right)\\
\biggl.&-2\frac{n-1}{n}\left((n-1)r-2n\right)\sharp \Ha{R}_{v}(p)\biggr\} \\
&-(n-1)r\left((n-1)r-2n\right)\sharp \Ov{K} \geq 0
\end{align*} 
in order to obtain the desired inequalities.
By $(\Ov{\Bb{D}},\Ov{\Gamma}_{p})_{\Ov{z}} \geq   \sharp \Ov{K}$, it is sufficient to show
\begin{align*}
2nr& \biggl\{\frac{n-1}{n}\left((n-1)r-2n\right)\frac{(\Ov{\Bb{D}}_{1},\Ov{\Gamma}_{p})_{\Ov{z}} }{2}
+\frac{(n^2-1)}{4n}r^{2}-n(r-1) \biggr.\\
&+\left( (n^2-1)\left[\frac{m}{n}\right](r-n\left[\frac{m}{n}\right])
-n(r-1)\right)\\
\biggl.&-2\frac{n-1}{n}\left((n-1)r-2n\right)\sharp \Ha{R}_{v}(p)\biggr\} \\
&-(n-1)r\left((n-1)r-2n\right)(\Ov{\Bb{D}}_{1},\Ov{\Gamma}_{p})_{\Ov{z}} \geq 0.
\end{align*} 
It is equivalent to 
\begin{align*}
2nr& \biggl\{\frac{(n^2-1)}{4n}r^{2}-n(r-1) \biggr.\\
&+\left( (n^2-1)\left[\frac{m}{n}\right](r-n\left[\frac{m}{n}\right])
-n(r-1)\right)\\
\biggl.&-2\frac{n-1}{n}\left((n-1)r-2n\right)\sharp \Ha{R}_{v}(p)\biggr\} \geq 0.
\end{align*} 
Since we have $\sharp \Ha{R}_{v}(p) \leq 3$,
it is sufficient to show 
\begin{align*}
\frac{(n^2-1)}{4n}r^{2}-n(r-1) 
+(n^2-1)(r-n)-n(r-1)-6\frac{n-1}{n}\left((n-1)r-2n\right) \geq 0.
\end{align*} 
We can check it by a simple calculation.
\end{proof}

\begin{prop}
\label{C4}
Assume $(C4)$ is hold, $r \geq 16$ if $n=4$ and $r \geq 15$ if $n=3$.
Then it holds that 
\begin{align*}
 2nr(r-1)K_{f}^{2}(\Ga_{p}) \geq 
  (n-1)r\left((n-1)r-2n\right)\sharp{\Ov{K}} .
\end{align*}
\end{prop}

\begin{proof}

\smallskip

Let $\Ov{z}$ be a singular point of $\Ov{R}$ on $\Ov{\Gamma}_{p}$.
Since $\Rm{mult}_{\Ov{z}}\Ov{R} \geq 3$, there exists a $\Ov{D} \in \Rm{L}_{\Rm{tr}}(\Ov{R}_{h},\Ov{z})$
which is not fixed by $\Rm{Stab}_{\Ov{K}}(\Ov{z})$.
We let $\Ov{\Bb{D}}$ be the $\Rm{Stab}_{\Ov{K}}(\Ov{z})$-orbit of $\Ov{D}$.
We introduce some notations as follows.
\begin{itemize}
\item $m_{\Ov{\Bb{D}}}:=\Rm{mult}_{\Ov{z}}\Ov{\Bb{D}}$,
\item $m:=\Rm{min}\{m' \mid m' \geq m_{\Ov{\Bb{D}}},\quad m' \in n\Z$ or $n\Z + 1\}$.
\end{itemize}

\smallskip

Case1.\;The fiber $\Ov{\Gamma}_{p}$ has no node at $\Ov{z}$.

Put $\Ov{K} \cdot \Ov{z} := \{ \Ov{z}_{1}(:=\Ov{z}), \cdots , \Ov{z}_{\Sh(\Ov{K}\cdot \Ov{z})}\}$.
Let $\Ha{P}\to \Ov{P}$ be the composite of blowing-ups at $\Ov{z}_{j}$ for $j=1,\cdots, \sharp \Ov{K}\cdot z$.
We denote by $\Ch{\psi}:\Ha{P}\to P$ the composite of the above $\Ha{P}\to \Ov{P}$ and $\Ov{P}\to P$.

\smallskip

(1-1)\:The case that $\Ov{\Gamma}_{p} \not\subset \Ov{R}$.

Suppose $\Rm{mult}_{\Ov{z}}\Ov{R} \in n\Z$.
Since $\sharp \Ha{R}_{v}(p)=0$, we have 
\begin{align*}
 &(r-1)K_{f}^{2}(\Ga_{p}) \geq\left((n^2 -1)\left[\frac{m}{n}\right]\left(r-n\left[\frac{m}{n}\right]\right) -n(r-1) \right)\sharp(\Ov{K}\cdot \Ov{z}).
 \end{align*}
from Lemma~\ref{addlem}.
It suffices to show 
\begin{align}
\label{1/22,1}
2nr& \left( (n^2-1)\left[\frac{m}{n}\right](r-n\left[\frac{m}{n}\right])-n(r-1)\right)
-(n-1)r\left((n-1)r-2n\right)m \geq 0
\end{align} 
in order to obtain the desired inequalities.
The left hand side of the above inequality can be considered as a quadratic function in $m$ defined in the interval $n \leq m \leq r/2$.
Since it is upward concave, it suffices to consider its value at the endpoints.
We can check it by a simple calculation.

Suppose $\Rm{mult}_{\Ov{z}}\Ov{R} \in n\Z+1$.
Let $\Ha{E}$ be an exceptional curve of $\Ov{z}$.
Since $\Ha{E} \subset \Ha{R}$, 
there exist a singular point of $\Ha{R}$ on $\Ha{E}$.
Since $\Ov{\Bb{D}}$ is the  $\Rm{Stab}_{\Ov{K}}(\Ov{z})$-orbit of $\Ov{D}$,
the proper transform of $\Ov{\Bb{D}}$ has a singular point of multiplicity $m_{\Ov{\Bb{D}}}$ on $\Ha{E}$.
Hence, we have
\begin{align*}
 &(r-1)K_{f}^{2}(\Ga_{p}) \geq \\
 & \left(-\frac{(n-1)}{n}\left((n-1)r-2n \right)B_{n}+
 (n^2-1)\left[\frac{m}{n}\right](r-n\left[\frac{m}{n}\right])-(r-1)n\right) 2\sharp( \Ov{K}\cdot \Ov{z} )
\end{align*}
from Lemma~\ref{lowerbdofK}.
We need to show 
\begin{align*}
2nr& \left\{-\frac{(n-1)}{n}\left((n-1)r-2n \right)B_{n} \right. \\
&\left. +(n^2-1)\left[\frac{m}{n}\right](r-n\left[\frac{m}{n}\right])-(r-1)n \right\}2\sharp( \Ov{K}\cdot \Ov{z})\\
&-(n-1)r\left((n-1)r-2n\right)\sharp \Ov{K} \geq 0
\end{align*} 
in order to obtain the desired inequality.
It is sufficient to show
\begin{align*}
2nr&\left\{-\frac{(n-1)}{n}\left((n-1)r-2n \right)B_{n} \right. \\
&\left. +(n^2-1)\left[\frac{m}{n}\right](r-n\left[\frac{m}{n}\right])-(r-1)n \right\}\\
&-(n-1)r\left((n-1)r-2n\right)\frac{m}{2} \geq 0.
\end{align*} 
We can check it by a simple calculation.
\smallskip

(1-2) The case that $\Ov{\Gamma}_{p} \subset \Ov{R}$.

(1-2-1) The fiber $\Ov{\Gamma}_{p}$ consists with one irreducible component.

By the assumption $\Ov{\Gamma}_{p} \subset \Ov{R}$,
there exist at least $n$-singular points $\Ov{R}$ on $\Ov{\Gamma}_{p}$.
Since $\Ov{K}$ has non-trivial stabilizer on an arbitrary singular point of $\Ov{R}$ on $\Ov{\Gamma}_{p}$,
$\Ov{K}$ is not cyclic group.
Suppose that there exist two singular points $\Ov{z}$ and $\Ov{z}'$ of $\Ov{R}$ 
such that $\Ov{K}\cdot \Ov{z} \cap \Ov{K}\cdot \Ov{z}' = \emptyset$. 
Changing the notation of  $\Ov{z}$ and $\Ov{z}'$ if necessary,
we may assume that $\sharp( \Ov{K}\cdot \Ov{z}) \geq \sharp( \Ov{K}\cdot \Ov{z}')$.
From the Table \ref{subgroup}, it holds that $\sharp (\Ov{K}\cdot \Ov{z}) \geq \sharp \Ov{K} /3$.
Hence we have 
\begin{align*}
 &(r-1)K_{f}^{2}(\Ga_{p}) \geq \\
 & \left(-\frac{(n-1)}{n}\left((n-1)r-2n \right)B_{n}+
 (n^2-1)(r-n)-(r-1)n\right) \sharp\Ov{K}/3
\end{align*}
from Lemma~\ref{lowerbdofK}.
Hence we have the desired inequality.

By the above argument, we may suppose that all singular points of $\Ov{R}$ on $\Ov{\Gamma}_{p}$
are on points $\Ov{K}\cdot \Ov{z}$.
Therefore it holds that $(\Ov{R}_{h},\Ov{\Gamma}_{p})_{\Ov{z}}\sharp(\Ov{K}\cdot \Ov{z}) = r$ .
Since  $\Rm{mult_{\Ov{z}}}\Ov{R}\in n\Z$ and $\sharp \Ha{R}_{v}(p)=1$,
we have 
\begin{align*}
 (r-1)K_{f}^{2}(\Ga_{p}) \geq 
& -2\frac{(n-1)}{n}\left((n-1)r-2n \right) \\
& +\left((n^2 -1)\frac{\Rm{mult}_{\Ov{z}}\Ov{R}}{n}\left(r-\Rm{mult}_{\Ov{z}}\Ov{R}\right) -n(r-1) \right)\sharp(\Ov{K}\cdot \Ov{z}).
 \end{align*}
from Lemma~\ref{addlem}.
It suffices to show 
\begin{align*}
2nr  &\left\{  -2\frac{(n-1)}{n}\left((n-1)r-2n \right) \right.\\
& \left. +\left((n^2 -1)\frac{\Rm{mult}_{\Ov{z}}\Ov{R}}{n}\left(r-\Rm{mult}_{\Ov{z}}\Ov{R}\right) -n(r-1) \right)\sharp(\Ov{K}\cdot \Ov{z}) \right\}\\
&-(n-1)r\left((n-1)r-2n\right)\sharp \Ov{K} \geq 0
 \end{align*}
 in order to obtain the desired inequality.

It is equivalent to 
\begin{align*}
 2nr&\biggl\{ -\frac{(n-1)}{n}\left((n-1)r-2n \right) \frac{2}{\sharp(\Ov{K}\cdot \Ov{z})}
 +\left((n^2 -1)\frac{\Rm{mult}_{\Ov{z}}\Ov{R}}{n}\left(r-\Rm{mult}_{\Ov{z}}\Ov{R}\right)  -n(r-1) \right) \biggr\}\\
 & -(n-1)r\left((n-1)r-2n\right)\sharp \Rm{Stab}_{\Ov{K}}(\Ov{z}) \geq 0.
 \end{align*}
By $\sharp \Rm{Stab}_{\Ov{K}}(\Ov{z}) \leq (\Ov{R}_{h},\Ov{\Gamma}_{p})_{\Ov{z}}\leq \Rm{mult}_{\Ov{z}}\Ov{R}$,
it is equivalent to 
\begin{align*}
 2nr&\biggl\{\left((n^2 -1)\frac{\Rm{mult}_{\Ov{z}}\Ov{R}}{n}\left(r-\Rm{mult}_{\Ov{z}}\Ov{R}\right)  -n(r-1) \right) \biggr\}\\
 & -(n-1)\left((n-1)r-2n\right)(r+4)\Rm{mult}_{\Ov{z}}\Ov{R}\geq 0.
 \end{align*}
Since there exist at least $n$ singular points of $\Ov{R}$ on $\Ov{\Gamma}_{p}$,
we have $\sharp (\Ov{K}\cdot \Ov{z}) \geq 3$.
Hence it holds $n \leq \Rm{mult}_{\Ov{z}}\Ov{R} \leq \frac{r}{3} +1$.
We can check the above inequality by a simple calculation.

\smallskip

(1-2-2) The fiber $\Ov{\Gamma}_{p}$ consists with two irreducible components.

Since $\sharp\Ha{R}_{v}(p)=2$, we have 
\begin{align*}
 (r-1)K_{f}^{2}(\Ga_{p}) \geq  
&-4\frac{(n-1)}{n}\left((n-1)r-2n \right) \\
+& \frac{(n^2-1)}{4n}r^{2}-n(r-1) \\
+& 2\left((n^2 -1)\left[\frac{m}{n}\right]\left(r-n\left[\frac{m}{n}\right]\right) -n(r-1) \right)
 \end{align*}
from Lemma~\ref{addlem}.
It is sufficient to show 
\begin{align*}
 2nr\biggl\{ &-4\frac{(n-1)}{n}\left((n-1)r-2n \right) 
+ \frac{(n^2-1)}{4n}r^{2}-n(r-1) \biggr. \\
\biggl. &+ 2\left((n^2 -1)\left[\frac{m}{n}\right]\left(r-n\left[\frac{m}{n}\right]\right) -n(r-1) \right) \biggr\}\\
 & -(n-1)r\left((n-1)r-2n\right)\sharp \Ov{K} \geq 0
 \end{align*}
to obtain the desired inequality.
By $2m \geq \Ov{K}$, it suffices to show 
\begin{align*}
 2nr\biggl\{ &-4\frac{(n-1)}{n}\left((n-1)r-2n \right) 
+ \frac{(n^2-1)}{4n}r^{2}-n(r-1) \biggr. \\
\biggl. &+ 2\left((n^2 -1)\left[\frac{m}{n}\right]\left(r-n\left[\frac{m}{n}\right]\right) -n(r-1) \right) \biggr\}\\
 & -(n-1)r\left((n-1)r-2n\right)2m \geq 0.
 \end{align*}
 Hence it suffices to show 
\begin{align*}
  &-4\frac{(n-1)}{n}\left((n-1)r-2n\right)+\frac{(n^2-1)}{4n}r^{2}-n(r-1)  \geq 0,\\
  &2nr \left((n^2 -1)\left[\frac{m}{n}\right]\left(r-n\left[\frac{m}{n}\right]\right) -n(r-1) \right) 
   -(n-1)r\left((n-1)r-2n\right)m \geq 0.
 \end{align*}
We can check it by a simple calculation and (\ref{1/22,1}).

\smallskip

Case2.\;$\Ov{\Gamma}_{p}$ has a node at $\Ov{z}$ which consists with two irreducible components 
$\Ov{\Gamma}_{1}$, $\Ov{\Gamma}_{2}$.

From the argument of Case1, we may assume 
that there exists only one singular point $\Ov{z}$ of $\Ov{R}$, where $\Ov{z}$ is the node of $\Ov{\Gamma}_{p}$.
Since $\Rm{L}_{\Rm{ta}}(\Ov{R},\Ov{z})=\emptyset$, it implies that $\Ov{\Gamma}_{p}\not\subset\Ov{R}$.

Suppose $\Rm{mult}_{\Ov{z}}\Ov{R} \in n\Z$.
Then we have 
\begin{align*}
 (r-1)K_{f}^{2}(\Ga_{p})  
&\geq \frac{n-1}{n}\left((n-1)r-2n\right)\frac{(\Ov{\Bb{D}}_{1},\Ov{\Gamma}_{p})_{\Ov{z}}}{2}\\
&+\frac{(n^2-1)}{4n}r^{2}-n(r-1)+\left( (n^2-1)\left[\frac{m}{n}\right](r-n\left[\frac{m}{n}\right])-n(r-1)\right)
\end{align*} 
from Lemma~\ref{addlem}.
We need to show 
\begin{align*}
2nr&\left\{ \frac{n-1}{n}\left((n-1)r-2n\right)\frac{(\Ov{\Bb{D}}_{1},\Ov{\Gamma}_{p})_{\Ov{z}}}{2}\right.\\
&\left.+\frac{(n^2-1)}{4n}r^{2}-n(r-1)+\left( (n^2-1)\left[\frac{m}{n}\right](r-n\left[\frac{m}{n}\right])-n(r-1)\right)\right\}\\
&-(n-1)r\left((n-1)r-2n\right) \sharp{\Ov{K}} \geq 0
\end{align*} 
to obtain the desired inequality.
Since the action $\kappa_{\Rm{sw}}$ may stable $\Ov{D}_{1}$, it holds $(\Ov{\Bb{D}}_{1},\Ov{\Gamma}_{p})_{\Ov{z}}\geq \sharp \Ov{K}/2$.
Hence it is sufficient to show
\begin{align*}
2nr&\left\{\frac{(n^2-1)}{4n}r^{2}-n(r-1)+\left( (n^2-1)\left[\frac{m}{n}\right](r-n\left[\frac{m}{n}\right])-n(r-1)\right)\right\}\\
&-(n-1)r\left((n-1)r-2n\right)(\Ov{\Bb{D}}_{1},\Ov{\Gamma}_{p})_{\Ov{z}} \geq 0.
\end{align*} 
By $2m \geq (\Ov{\Bb{D}}_{1},\Ov{\Gamma}_{p})_{\Ov{z}}$, it suffices to show
\begin{align*}
2nr&\left\{\frac{(n^2-1)}{4n}r^{2}-n(r-1)+\left( (n^2-1)\left[\frac{m}{n}\right](r-n\left[\frac{m}{n}\right])-n(r-1)\right)\right\}\\
&-(n-1)r\left((n-1)r-2n\right)2m \geq 0.
\end{align*} 
We can check it by a simple calculation.

Suppose $\Rm{mult}_{\Ov{z}}\Ov{R} \in n\Z+1$.
Since the exceptional curve over $\Ov{z}$ is contained in the branch locus,
there exist at least $(n-1)$ singular points of $\Ov{R}$ on it.
Hence we have 
\begin{align*}
 &(r-1)K_{f}^{2}(\Ga_{p})  \\
&\geq \frac{n-1}{n}\left((n-1)r-2n\right)\frac{(\Ov{\Bb{D}}_{1},\Ov{\Gamma}_{p})_{\Ov{z}}}{2}
-2\frac{(n-1)}{n}\left((n-1)r-2n \right)\\
&+\left(-2\frac{(n-1)}{n}\left((n-1)r-2n \right) +(n^2-1)(r-n)-n(r-1) \right)(n-1)\\
&+\frac{(n^2-1)}{4n}r^{2}-n(r-1)+\left( (n^2-1)\left[\frac{m}{n}\right](r-n\left[\frac{m}{n}\right])-n(r-1)\right)\\
&\geq \frac{n-1}{n}\left((n-1)r-2n\right)\frac{(\Ov{\Bb{D}}_{1},\Ov{\Gamma}_{p})_{\Ov{z}}}{2}\\
&+\frac{(n^2-1)}{4n}r^{2}-n(r-1)+\left( (n^2-1)\left[\frac{m}{n}\right](r-n\left[\frac{m}{n}\right])-n(r-1)\right)
\end{align*} 
by Lemma~\ref{addlem} and a simple calculation.
Hence we can show it by a similar argument in the case of $\Rm{mult}_{\Ov{z}}\Ov{R} \in n\Z$.
\end{proof}

From Proposition \ref{et.case}, \ref{sm.case}, \ref{C1}, \ref{C2}, \ref{C3} and \ref{C4},
we obtain the following.

\begin{prop}
\label{summarize}
Let $\Ov{\varphi}:\Ov{P} \to B$ be the model of Lemma~\ref{hyoujyunka}.
Assume $r \geq 16$ if $n=4$ and $r \geq 18$ if $n=3$.
Take an arbitrary fiber $\Ov{\Gamma}_{p}$ of $\Ov{\varphi}$ with $K_{f}^{2}(\Gamma_{p})>0$.
If either $\Ov{R}$ has singular point on $\Ov{\Gamma}_{p}$, 
$\Ov{\varphi}|_{\Ov{R}_{h}}$ has a ramification point on $\Ov{\Gamma}_{p}$
or $\Ov{K}\cong \Z_{2}$,
it holds 
\begin{align*}
2n(r-1)K_{f}^{2}(\Ga_{p})\geq (n-1)\left((n-1)r-2n) \right) \Sh \Ov{K}.
\end{align*}
Otherwise it holds that 
\begin{align*}
 4nr(r-1)K_{f}^{2}(\Ga_{p}) \geq  \left((n^2-1)r^{2}-4n^{2}(r-1)\right)\sharp \Ov{K}.
\end{align*}
\end{prop}

\section{Upper bound of the automorphism group of fibered surface}

\noindent

Now, we consider the upper bound of the order of the automorphism group of $f$.
Let $\Ga_{p}$ be a fiber which is tangent to $R_{h}$ or passes through a singular point of $R$ or is over $p \in \Delta$.
Put $r_{p}=\Sh \Stab_{H}(p)$.
From $r \geq 3n$ and Lemma~\ref{lowerbdofK},
we have
\begin{align*}
K_{f}^2 \geq \frac{\Sh H}{r_{p}}K_{f}^{2}(\Ga_{p}).
\end{align*}
Since $\Sh G=\Sh K\cdot \Sh H=n\Sh \KB\cdot \Sh H$, we get
\begin{align*}
\Sh G \leq \frac{n\Sh \KB}{K_{f}^{2}(\Ga_{p})}r_{p}K_{f}^2.
\end{align*}
From Proposition~\ref{Rhassing} and Proposition~\ref{summarize},
there exists a fiber $\Gamma_{p}$ such that
\begin{align*}
\frac{n\Sh \KB}{K_{f}^{2}(\Ga_{p})}\leq \mu_{r,n}
\end{align*}
and, hence, 
\begin{align*}
\Sh G \leq \mu_{r,n}r_{p}K_{f}^2.
\end{align*}
So we only have to manage $r_{p}$.

\begin{thm}
\label{up.bd.1}
Let $f:S \to B$ be a primitive cyclic covering fibration of type $(g,0,n)$
with $r \geq 16$ if $n=4$ and $r \geq 18$ if $n=3$.
Put 
\begin{align*}
\mu_{r,n}:=\frac{2n^{2}(r-1)}{(n-1)((n-1)r-2n)},\quad \mu_{r,n}':=\frac{4n^{2}r(r-1)}{(n^2 -1)r^2-4n^2(r-1)}.
\end{align*}
Assume furthermore that $f$ is not locally trivial and, when $g(B)=0$, $f$ has at least $3$ singular fibers (e.g., $f$ is not iso-trivial).
Let $G$ be a finite subgroup of $\Aut (f)$.
Then it holds
\[
  \Sh G \leq \begin{cases}
    6(2g(B)-1)\mu_{r,n}K_{f}^2 & (g(B) \geq 1 ) \\
     5\mu_{r,n}' K_{f}^2 & (g(B)=0) .
  \end{cases}
\]
\end{thm}

\begin{proof}
Put $r_{p}=\Sh \Stab_{H}(p)$ for $p \in B$.
Let $\pi:B\to B/H$ be the quotient map of $B$ by $H$.
Note that $r_{p}$ be the ramification index of $\pi$ at $p$.

\smallskip

(i) The case of $g(B)\geq 2$.

\smallskip

We denote the genus of $B/H$ by $g(B/H)$.
From the Hurwitz formula, we get
\begin{align*}
2g(B)-2=\Sh H \left(2g(B/H)-2+\sum_{i=1}^s \frac{r_{i}-1}{r_{i}}\right), 
\end{align*}
where $s$ is the number of  ramification points and $r_{i}$ is the ramification index.
Put
\begin{align*}
T:=2g(B/H)-2+\sum_{i=1}^{s} \frac{r_{i}-1}{r_{i}},
\end{align*}
which must be positive in the present case.
If $g(B/H)\geq 2$, then we get $\Sh H\leq g(B)-1$ by $T\geq 2$, and 
it follows that $r_{i}\leq \Sh H \leq g(B)-1$ for any $i=1, \cdots s$.

Assume that $g(B/H)=1$. Then we get  $s>0$ by $T>0$.
By $r_{i}\geq 2$ for any $i=1, \cdots s$,
we get
$1-1/r_{i}\geq 1/2.$
Therefore we obtain 
$r_{i}\leq \Sh H \leq 4(g(B)-1)$
for any $i=1, \cdots s$ by $T \geq 1/2$. 

Assume that $g(B/H)=0$.
When $s\geq 5$, we get 
$r_{i}\leq \Sh H \leq 4(g(B)-1)$
for any $i=1, \cdots s$ by $T\geq 1/2$.
When $s=4$, one of $r_{i}$ is not less than $3$.
So we get 
$r_{i}\leq \Sh H \leq 12(g(B)-1)$
for any $i=1, \cdots s$ by $T \geq 1/6$. 
When $s=3$, we may assume $r_{1}\geq r_{2}\geq r_{3}$.
By the definition of $T$,
we get  
\begin{align*}
r_{1}\leq \Sh H=\frac{2g(B)-2}{1-\frac{1}{r_{1}}-\frac{1}{r_{2}}-\frac{1}{r_{3}}},
\end{align*}
so we obtain 
\begin{align*}
r_{1}-1-\frac{r_{1}}{r_{2}}-\frac{r_{1}}{r_{3}} \leq 2g(B)-2.
\end{align*}
If $r_{3}=2$, then  $r_{2}\geq 3$, so we get 
\begin{align*}
r_{1}-1-\frac{r_{1}}{3}-\frac{r_{1}}{2} \leq 2g(B)-2.
\end{align*}
Hence we get
$r_{1}\leq 6(2g(B)-1)$.
If $r_{2}\geq r_{3}\geq 3$, we get
$r_{1}\leq 3(2g(B)-1)$
Therefore we obtain 
$r_{p}\leq 6(2g(B)-1)$
for any $p\in B$.

\smallskip

(ii) The case of $g(B)=1$.

We need not consider a translation, since it has no fixed points.
Then it is well known that the order of the automorphism group of $B$ which fixes a point is at most $6$.
So the order of stabilizer of $H$ is at most $6$, and we get $r_{p}\leq 6$.

\smallskip 

(iii) The case of $g(B)=0$

If $H$ is neither a cyclic group nor a dihedral group, the order of a stabilizer of $H$ is at most $5$.
So we may assume $H$ is either a cyclic group or a dihedral group.
It is well known that a rational pencil with at most two singular fibers is iso-trivial.
If the number of singular fibers of $f$ is at least three, then 
there exist a point $p \in B$ such that 
\begin{align*}
K_{f}^{2}(\Ga_{p})>0,\quad r_p \leq 5, \quad
\frac{n\Sh \KB}{K_{f}^{2}(\Ga_{p})}\leq \frac{4n^{2}r(r-1)}{(n^2 -1)r^2-4n^2(r-1)}.
\end{align*}
from Proposition~\ref{summarize}.
Hence we obtain the desired inequality.
\end{proof}

\begin{thm}
\label{6/16.cor}
Let $f:S \to \Bb{P}^{1}$ be a non-locally trivial primitive cyclic covering fibration of type $(g,0,n)$
with $r \geq 16$ if $n=4$ and $r \geq 18$ if $n=3$.
Then it holds
\begin{align*}
\Sh \Aut (S/B) \leq \left(1+\frac{1}{n-1}\right)^{2}\frac{2g+n-1}{2g-2} K_{f}^2.
\end{align*}
\end{thm}

\begin{proof}
Since $f$ is not a locally trivial fibration,
there exist at least two singular fibers of $f$.
Hence we have at least two fibers of $\varphi$ such that $K_{f}^{2}(\Gamma)>0$.
We apply Proposition~\ref{summarize} for $G=\Rm{Aut}(S/B)$.
Then we have $ \sharp  \Rm{Aut}(S/B)=n \sharp  \Ov{K}$.
If $\Delta = \emptyset$ or $\Ov{K} \cong \Z_{2}$, it holds clearly from Proposition~\ref{summarize}.
Hence we may assume that $\Delta \neq \emptyset$,
so in particular $\Ov{K}$ is an dihedral group and $\sharp \Delta \geq 2$ .

\smallskip

Case~1.\;There exist at least three fibers with $K_{f}^{2}(\Gamma)>0$.

\smallskip

From Lemma~\ref{Rhassing} and Proposition~\ref{summarize}, 
there exists a fiber $\Gamma$ such that 
\begin{align*}
2n(r-1)K_{f}^{2}(\Gamma)\geq (n-1)\left((n-1)r-2n) \right) \Sh \Ov{K}.
\end{align*}
Hence we obtain the desired inequality.

\smallskip

Case~2.\;There exist at least just two fibers with $K_{f}^{2}(\Gamma)>0$.

\smallskip

We note that $\sharp \Delta = 2$ in this case.
Let $\Gamma_{p_{i}}$ be a fiber of $\Ov{\varphi}$ 
such that $K_{f}^{2}(\Gamma_{p_{i}})>0$ for $i=1,2$.

Assume that there exists a irreducible component $C\subset \Ov{R}_{h}$ such that
$\Ti{C}\to \Bb{P}^{1}$ has a ramification point, where $\Ti{C} \to \Bb{P}^{1}$ is the composition of the normalization $\Ti{C} \to C$ and $\Ov{\varphi}|_{C}:C \to \Bb{P}^{1}$.
Since $\Ti{C}\to \Bb{P}^{1}$ at least two critical values,
$\Ov{\varphi}|_{R_{h}}:R_{h} \to B$ has a ramification point 
or $\Ov{R}_{h}$ has a singular point on $\Gamma_{p_{1}}$ and $\Gamma_{p_{2}}$.
Therefore we have 
\begin{align*}
2n(r-1)K_{f}^{2}(\Gamma_{p_{i}})\geq (n-1)\left((n-1)r-2n) \right) \Sh \Ov{K}.
\end{align*}
for $i=1,2$.
Hence we obtain the desired inequality.

Assume that the morphism $\Ti{C}\to \Bb{P}^{1}$ is an unramified covering for any irreducible component $C \subset \Ov{R}_{h}$,  where $\Ti{C} \to \Bb{P}^{1}$ is the composition of the normalization $\Ti{C} \to C$ and $\Ov{\varphi}|_{C}:C \to \Bb{P}^{1}$.
Changing the numbering of $\Ov{\Gamma}_{p_{1}}$ and $\Ov{\Gamma}_{p_{2}}$ if necessary,
we may assume that $\Ov{R}$ is smooth and $\Ov{\varphi}|_{R_{h}}:R_{h} \to \Bb{P}^{1}$ is unramified 
locally around $\Ov{\Gamma}_{p_{1}}$.
Hence we have 
\begin{align*}
 4nr(r-1)K_{f}^{2}(\Gamma_{p_{1}}) \geq  \left((n^2-1)r^{2}-4n^{2}(r-1)\right)\sharp \Ov{K}.
\end{align*}
Let $\Ov{z}$ be a singular point of $\Ov{R}$ on $\Ov{\Gamma}_{p_{2}}$.
We note that $\sharp(\Ov{K}\cdot\Ov{z})=2$ or $\sharp \Ov{K}$ from the Table~\ref{subgroup}.
If $\sharp(\Ov{K}\cdot\Ov{z}) = \sharp \Ov{K}$,
then we have 
\begin{align*}
 &(r-1)K_{f}^{2}(\Gamma_{p_{2}}) \geq  \\
&\left(-\frac{(n-1)}{n}\left((n-1)r-2n \right) B_{n}
+ \left((n^2 -1)(r-n) -n(r-1) \right)\right) \sharp \Ov{K}.
 \end{align*}
from Lemma~\ref{lowerbdofK}.
Hence we get 
\begin{align*}
 &(r-1)K_{f}^{2}(\Gamma_{p_{2}}) \geq  \frac{3}{4}\left(1-\frac{1}{r}\right)\left((n-1)r-2n \right) \sharp \Ov{K}.
 \end{align*}
by a simple calculation,
we obtain the desired inequality.
If $\Ov{K}\cdot\Ov{z} = \{ \Ov{z}_{1},\Ov{z}_{2}\}$,
let $\Ha{P}\to \Ov{P}$ be the composite of blowing-ups at $\Ov{z}_{1}$ and $\Ov{z}_{2}$.
Since $\Ov{R}_{h}$ does not passes through the node of $\Ov{\Gamma}_{p_{2}}$ from the assumption,
it implies that $\Ov{\Gamma}_{p_{2}}\not\subset \Ov{R}$.
Note also that $m:=\Rm{mult}_{\Ov{z}}\Ov{R}\in n\Z$ and 
$m = (\Ov{R},\Ov{\Gamma}_{p})_{\Ov{z}} \geq \sharp \Ov{K}/2$ from the assumption.
Since $\sharp\Ha{R}_{v}(p)=0$, we have 
\begin{align*}
 (r-1)K_{f}^{2}(\Ga_{p}) &\geq 
\frac{(n^2-1)}{4n}r^{2}-n(r-1)+2\left( (n^2-1)\left[\frac{m}{n}\right](r-n\left[\frac{m}{n}\right])-n(r-1)\right)\\
& \geq  \frac{3}{4}\left(1-\frac{1}{r}\right)\left((n-1)r-2n \right) \sharp \Ov{K},
\end{align*} 
we obtain the desired inequality.
\end{proof}

\section{An example}

We construct a primitive cyclic covering fibration $f:S \to \Bb{P}^{1}$ of type $(g,0,n)$ with
\begin{align*}
\sharp \Rm{Aut}(S/B) \geq \left(1+\frac{1}{n-1}\right)\left(1+\frac{n-1}{2g+n-1}\right)\frac{2g+n-1}{2g-2}.
\end{align*}

We consider the projection $\Phi:\Bb{P}^2 \times \C \to \C$ defined by $\Phi ([w_0:w_1:w_2],z)=z$ where
$[w_0:w_1:w_2]$ is a system of homogeneous coordinates on $\Bb{P}^2$ and $z$ is a coordinate of $\C$.
Let $P^{\circ}$ be a hypersurface of $\Bb{P}^2 \times \C$ defined by the equation $w_0 w_1 -w_2^2 z=0$.
We consider the morphism $\PHI^{\circ}: P^{\circ} \hookrightarrow \Bb{P}^2 \times \C \xrightarrow{\Phi} \C$,
then $\PHI^{\circ}$ is a conic bundle over $\C$.
We have the commutative diagram:
\[
\xymatrix{
  P^{\circ} \ar[d]_{\PHI^{\circ}} \ar@{^{(}-_{>}}[r] & \ar[d]_{\Phi} \Bb{P}^2 \times \C \\
 \C \ar[r]^{\cong}& \C     \\	
}	
\]
We denote the fiber of $\PHI^{\circ}$ over $z \in \C$ by $\Gamma_{z}^{\circ}$.
Note that $\varphi^{\circ}$ has only one singular fiber $\Gamma_{0}^{\circ}$ defined by the equation $w_0 w_1 =0$ in $\Bb{P}^2 $, which is the fiber of $\Phi$ over $0 \in \C$.

Let $l$ be a positive integer divisible by $n$.
We define automorphisms of $\PHI^{\circ}: P^{\circ} \to \C$ as follows:
\begin{align*}
 &\zeta:([w_0:w_1:w_2],z) \mapsto ([\zeta_{l}w_0: \zeta_{l}^{-1} w_1:w_2],z),\\
 & \kappa_{\Rm{sw}}:([w_0:w_1:w_2],z) \mapsto ([w_1:w_0:w_2],z),
\end{align*}
where $\zeta_l$ denotes a primitive $l$-th root of unity.
Let $\Ov{K}$ be the automorphism group on $P^{\circ}$ generated by $\zeta$ and $\kappa_{\Rm{sw}}$,
then $\Ov{K}$ is a dihedral group with order $2l$. 
Since $\kappa_{\Rm{sw}}$ switches irreducible components of $\Gamma_{0}^{\circ}$,
$\kappa_{\Rm{sw}}$ does not preserve a relatively minimal model of $\PHI^{\circ}$.

Let $H$ be a hyper plane of $\Bb{P}^{2}\times \C$ defined 
by equation $w_{0}=w_{1}$ and let $D^{\circ}=P^{\circ} \cap H$.
We can check that $D^{\circ}$ is smooth irreducible curve, 
passes through the node of $\Gamma_{0}^{\circ}$,
$\PHI^{\circ}|_{D^{\circ}}:D^{\circ} \to \C$ is a map of degree 2,  
and $\kappa^{\ast}_{\Rm{sw}}D^{\circ}=D^{\circ}$.
Let $R^{\circ}:=\sum_{i=1}^{l/2} (\zeta^{i})^{\ast}D^{\circ}$.
Then $R^{\circ}$ has one singular point at $([0:0:1],0)$ of multiplicity $l/2$
and is $\Ov{K}$-stable.
We consider the natural injective map  $ \Bb{P}^2 \times \C \hookrightarrow \Bb{P}^2\times \Bb{P}^{1}$
defined by $ ([w_0:w_1:w_2],z)\mapsto ([w_0:w_1:w_2],[z:1])$.
Let $\Ov{P}^{\circ}$ be the closure of $P^{\circ}$ in $\Bb{P}^{2}\times \Bb{P}^{1}$
and $\Ov{\varphi}^{\circ}:\Ov{P}^{\circ}\to \Bb{P}^{1}$ be the natural ruling.
We have the commutative diagram
\[
\xymatrix{
  P^{\circ} \ar[d]_{\PHI^{\circ}} \ar@{^{(}-_{>}}[r] & \ar[d]_{\Ov{\varphi}^{\circ}} \Ov{P}^{\circ} \\
 \C \ar@{^{(}-_{>}}[r]& \Bb{P}^{1}.     \\	
}	
\] 
Then $\Ov{P}^{\circ}$ has a two $A^{1}$-type singular points at $([0:1:0],[1:0])$ and $([1:0:0],[1:0])$.
The group $\Ov{K}$ acts $\Ov{P}^{\circ}$ over $\Bb{P}^{1}$ 
and the action $\kappa_{\Rm{sw}}$ switches these singular points.
Let $\Ov{D}^{\circ}$ be the closure of $D^{\circ}$ in $\Ov{P}^{\circ}$ and
let $\Ov{R}^{\circ}:=\sum_{i=1}^{l/2} (\zeta^{i})^{\ast}\Ov{D}^{\circ}$.
The divisor $\Ov{R}^{\circ}$ has never passes through two singular points of $\Ov{P}^{\circ}$
and $\Ov{R}^{\circ}$ has one singular point at $([0:0:1],[0:1])$ of multiplicity $l/2$.
Let $\Ti{P}\to \Ov{P}^{\circ}$ be the composition of blowing ups at two singular points of $\Ov{P}^{\circ}$
and the singular point of $\Ov{R}^{\circ}$.
The smooth projective surface $\Ti{P}$ admits the natural ruling $\Ti{\varphi}:\Ti{P}\to \Bb{P}^1$ and 
$\Ov{K} \subset \Rm{Aut}(\Ti{P}/\Bb{P}^{1})$.
We denote by $\Ti{R}$ the proper transform of $\Ov{R}^{\circ}$ to $\Ti{P}$.
Since $(\zeta^{i})^{\ast}\Ov{D}^{\circ}\sim(\zeta^{j})^{\ast}\Ov{D}^{\circ}$ for any $i,j$ 
and $l/2$ is divisible by $n$,
there exist a divisor $\Ti{\Fr{d}}$ on $\Ti{P}$ such that $\Ti{R} \in |n\Ti{\Fr{d}}|$.
We consider the classical cyclic $n$-covering 
\begin{align*}
\Ti{\theta}:\Ti{S}:=
\mathrm{Spec}_{ \Ti{P}}\left(\bigoplus_{j=0}^{n-1} \mathcal{O}_{\Ti{P}}(-j\Ti{\Fr{d}})\right)\to \Ti{P}
\end{align*}
branched along $\Ti{R}$.
Then $\Ti{S}$ admits the relatively minimal fibration $\Ti{f}:\Ti{S}\to \Bb{P}^{1}$.
Hence we rewrite $\Ti{f}:\Ti{S}\to \Bb{P}^{1}$ as $f:S \to \Bb{P}^{1}$.
Since it holds $\kappa_{\Ti{P}}^{\ast}\Fr{d} \sim \Fr{d}$ and $\kappa_{\Ti{P}}^{\ast}\Ti{R}=\Ti{R}$
for any $\kappa_{\Ti{P}} \in \Ov{K}$,
there exist an action $\kappa_{S} \in \Rm{Aut}(\Bb{P}^{1})$ for any $\kappa_{\Ti{P}} \in \Ov{K}$ 
which commutes the diagram 
\begin{equation*}
\xymatrix{
S\ar[r]^-{\kappa_{S}}\ar[d]_-{\Ti{\theta}}\ar@{}[rd]|{\circlearrowright}&S\ar[d]^-{\Ti{\theta}}\\
 \Ti{P}\ar[r]_-{\kappa_{\Ti{P}}}& \Ti{P}.
}
\end{equation*}
Hence we have $\sharp \Rm{Aut}(S/\Bb{P}^{1}) \geq 2nl$.

Let $\varphi:P \to \Bb{P}^{1}$ be a relatively minimal model of $\Ti{\varphi}$,
$R$ is the image of $\Ti{R}$ to $P$,
$\Gamma_{p}$ be a fiber of $\varphi$ over $p \in \Bb{P}^{1}$ and $r:=R\Gamma_{p}$.
By the construction, we have 
\begin{align*}
 &(r-1)K_{f}^{2}(\Gamma_{[1:0]})\\
&=\frac{n-1}{n}\left((n-1)r-2n \right)\frac{l}{2}
+2\left(\frac{n^{2}-1}{n}\frac{l}{2}\left(r-\frac{l}{2}\right) -n(r-1) \right),\\
&(r-1)K_{f}^{2}(\Gamma_{[0:1]})\\
&=\frac{n-1}{n}\left((n-1)r-2n \right)\frac{l}{2}
+2\left(\frac{n^{2}-1}{n}\frac{l}{2}\left(r-\frac{l}{2}\right) -n(r-1) \right).
\end{align*}
By $r =l$, we get 
\begin{align*}
(r-1)K_{f}^{2}&=(r-1)K_{f}^{2}(\Gamma_{[1:0]})+(r-1)K_{f}^{2}(\Gamma_{[0:1]})=2\left((n-1)r-2n \right)(r-1).
\end{align*}
Hence we have
\begin{align*}
\sharp \Rm{Aut}(S/\Bb{P}^{1})& \geq \frac{nr}{(n-1)r-2n }K_{f}^{2} \\
&=\left(1+\frac{1}{n-1}\right)\left(1+\frac{n-1}{2g+n-1}\right)\frac{2g+n-1}{2g-2}K_{f}^{2}.
\end{align*}

\vspace{2\baselineskip}
{}

\bigskip
\bigskip

Hiroto Akaike

Department of Mathematics, 
Graduate School of Science, 
Osaka University,

1-1 Machikaneyama, Toyonaka, Osaka 
560-0043, Japan

e-mail: u802629d@ecs.osaka-u.ac.jp
\end{document}